\documentclass[letterpaper]{amsart}
\pdfoutput=1
\newtheorem{theorem}{Theorem}[section]
\newtheorem{lemma}[theorem]{Lemma}

\newtheorem{corollary}[theorem]{Corollary}

\theoremstyle{definition}
\newtheorem{definition}[theorem]{Definition}
\newtheorem{example}[theorem]{Example}

\theoremstyle{remark}
\newtheorem{remark}[theorem]{Remark}
\usepackage{hyperref}
\usepackage{graphicx}
\usepackage{color}
\usepackage{amsmath}
\usepackage{amsfonts}
\usepackage{euler}
\usepackage{amssymb}
\usepackage{epsfig}
\usepackage{rotating}
\usepackage{graphics}
\usepackage{dynkin-diagrams}
\usepackage{amsrefs}

\newcommand{\dz}{\wedge}

\newcommand{\ba}{\begin{array}}

\newcommand{\ea}{\end{array}}

\newcommand{\beq}{\begin{eqnarray}}
\newcommand{\eeq}{\end{eqnarray}}

\newcommand{\der}{{\rm d}}

\DeclareMathSymbol{\minu}{\mathbin}{AMSa}{"39}

\numberwithin{equation}{section}

\usepackage{t1enc}
\def\frak{\mathfrak}

\newcommand{\bbC}{\mathbb{C}}

\usepackage{amssymb}
\usepackage{amscd}

\newcommand{\hook}{\raisebox{-0.35ex}{\makebox[0.6em][r]
{\scriptsize $-$}}\hspace{-0.15em}\raisebox{0.25ex}{\makebox[0.4em][l]{\tiny
 $|$}}}

\newcommand{\bma}{\begin{pmatrix}}
\newcommand{\ema}{\end{pmatrix}}

\def\bbZ{{\mathbb{Z}}}

\usepackage{ifthen}

\newcommand{\Span}{\mathrm{Span}}

\def\bean{\begin{eqnarray}}
\def\eean{\end{eqnarray}}
\def\bea{\begin{eqnarray*}}
\def\eea{\end{eqnarray*}}
\def\be{\begin{equation}}
\def\ee{\end{equation}}
\def\beq*{\begin{equation*}}
\def\eeq*{\end{equation*}}
\def\bal{\begin{align*}}
\def\eal{\end{align*}}
\def\baln{\begin{align}}
\def\ealn{\end{align}}

\def\beg{\begin{gather*}}
\def\eng{\end{gather*}}
\def\bqu{\begin{question}}
\def\equ{\end{question}}

\newcommand{\roa}{{\stackrel{\scriptscriptstyle{a}}{\rho}}{}}

\newcommand{\bbR}{\mathbb{R}}
\newcommand{\sog}{\mathbf{SO}}

\newcommand{\glg}{\mathbf{GL}}

\newcommand{\spg}{\mathbf{Sp}}

\newcommand{\soa}{\frak{so}}
\newcommand{\spa}{\frak{sp}}

\newcommand{\gla}{\frak{gl}}
\newcommand{\sla}{\frak{sl}}

\begin{document}
\title[Exceptional geometries]{Exceptional simple real Lie algebras $\mathfrak{f}_4$ and $\mathfrak{e}_6$ via contactifications}
\vskip 1.truecm
\author{Pawe\l~ Nurowski} \address{Center for Theoretical Physics,
Polish Academy of Sciences, Al. Lotnik\'ow 32/46, 02-668 Warszawa, Poland}
\email{nurowski@cft.edu.pl}
\thanks{The research was funded from the Norwegian Financial Mechanism 2014-2021 with project registration number 2019/34/H/ST1/00636.}

\date{\today}
\begin{abstract}
  In Cartan's PhD thesis, there is a formula defining a certain rank 8 vector distribution in dimension 15, whose algebra of authomorphism is the split real form of the simple exceptional complex Lie algebra $\mathfrak{f}_4$. Cartan's formula is written in the standard Cartesian coordinates in $\bbR^{15}$. In the present paper we explain how to find analogous formula for the flat models of any bracket generating distribution $\mathcal D$ whose symbol algebra $\mathfrak{n}({\mathcal D})$ is constant and 2-step graded, $\mathfrak{n}({\mathcal D})=\mathfrak{n}_{-2}\oplus\mathfrak{n}_{-1}$.

  The formula is given in terms of a solution to a certain system of linear algebraic equations determined by two representations $(\rho,\mathfrak{n}_{-1})$ and $(\tau,\mathfrak{n}_{-2})$ of a Lie algebra $\mathfrak{n}_{00}$ contained in the $0$th order Tanaka prolongation $\mathfrak{n}_0$ of $\mathfrak{n}({\mathcal D})$.

  Numerous examples are provided, with particular emphasis put on the distributions with symmetries being real forms of simple exceptional Lie algebras $\mathfrak{f}_4$ and $\mathfrak{e}_6$.
  
\end{abstract}
\maketitle
\tableofcontents
\newcommand{\gat}{\tilde{\gamma}}
\newcommand{\Gat}{\tilde{\Gamma}}
\newcommand{\thet}{\tilde{\theta}}
\newcommand{\Thet}{\tilde{T}}
\newcommand{\rt}{\tilde{r}}
\newcommand{\st}{\sqrt{3}}
\newcommand{\kat}{\tilde{\kappa}}
\newcommand{\kz}{{K^{{~}^{\hskip-3.1mm\circ}}}}
\newcommand{\di}{{\rm div}}
\newcommand{\curl}{{\rm curl}}
\newcommand{\tn}{{\mathcal N}}
\newcommand{\ten}{{\Upsilon}}
\newcommand{\invol}[2]{\draw[latex-latex] (root #1) to
  [out=-30,in=-150] (root #2);}
\newcommand{\invok}[2]{\draw[latex-latex] (root #1) to
[out=-90,in=-90] (root #2);}
\section{Introduction: the notion of a contactification}\label{intr}
A \emph{contact structure} $(M,{\mathcal D})$ on a $(2n+1)$ dimensional real manifold $M$ is usually defined in terms of a 1-form $\lambda$ on $M$ such that
$$\underbrace{\der\lambda\dz\der\lambda\dz\dots\dz\der\lambda}_{n\,\,\mathrm{times}}\dz\lambda\neq 0$$
at each point $x\in M$. Given such a 1-form, the contact structure $(M,{\mathcal D})$ on $M$ is the rank $s=2n$ \emph{vector distribution}
$${\mathcal D}=\{X\in \mathrm{T}M\,\,\mathrm{s.t.}\,\, X\hook\lambda=0\}.$$
Note that any $\lambda'=a\lambda$, with $a$ being a nonvanishing function on $M$, defines the same contact structure $(M,{\mathcal D})$. We also note that given a contact structure $(M,{\mathcal D})$, we additionally have a family of 2-forms on $M$
$$\omega'=a\omega +\mu\dz\lambda,\quad\mathrm{with}\quad \omega=\der\lambda,$$
where $a\neq 0$ is a function, and $\mu$ is a 1-form on $M$. This, in particular, means that given a contact structure $(M,{\mathcal D})$, we have a rank $s=2n$ (bracket generating) distribution ${\mathcal D}$, and a \emph{line} of a  closed 2-form $\omega$ \emph{in the distribution} ${\mathcal D}$, with
$$\der\omega=0\quad\&\quad\underbrace{\omega\dz\omega\dz\dots\dz\omega}_{n\,\,\mathrm{times}}\neq 0.$$

This can be compared with the notion of a \emph{symplectic} structure $(N,[\omega])$ on a $s=2n$ dimensional real manifold $N$. Such a structure is defined in terms of a line $\omega'=h\omega$ of a nowhere vanishing 2-form $\omega$ on $N$, such that
$$\der\omega=0\quad\&\quad\underbrace{\omega\dz\omega\dz\dots\dz\omega}_{n\,\,\mathrm{times}}\neq 0.$$
Here, contrary to the contact case, we have a \emph{ line} of a closed 2-form $\omega$ \emph{in the tangent space} $\mathrm{T}N$ rather than in the proper vector subbundle ${\mathcal D}\subsetneq\mathrm{T}N$.

By the \emph{Poincar\'e lemma}, locally, in an open set ${\mathcal O}\subset N$, the form $\omega$ defines a 1-form $\Lambda$ on $N$ such $\der\Lambda=\omega$. Therefore given a symplectic structure $(N,[\omega])$, we can locally \emph{contactify} it, by considering a $(2n+1)$ dimensional manifold $${\mathcal U}=\bbR\times{\mathcal O}\stackrel{\pi}{\to}\mathcal O,$$ with a 1-form $$\lambda=\der u+\pi^*(\Lambda) $$
on $\mathcal U$; here the real variable $u$ is a coordinate along the $\bbR$ factor in $\mathcal U=\bbR\times\mathcal O$.  As a result the structure $(M,{\mathcal D})=\big({\mathcal U},\ker(\lambda)\big)$ is a \emph{contact structure}, called a \emph{contact structure associated with the symplectic structure} $(N,[\omega])$.
\vspace{0.3cm}

We introduce the notion of a \emph{contactification} as a generalization of the above considerations.  

\begin{definition}\label{def1}
  Let $N$ be an $s$-dimensional manifold and let
  $\der{\mathcal D}^\perp:=\Span(\omega^1,\omega^2,$ $\dots,\omega^r)$ be a rank $r$ subbundle of $\bigwedge^2N$. Consider an $(s+r)$-dimensional fiber bundle $F\to M\stackrel{\pi}{\to}N$ over $N$. Let $(X_1,X_2,\dots, X_r)$ be a coframe of \emph{vertical vectors} in $M$. In particular we have $\pi_*(X_i)=0$ for all $i=1,2,\dots,r$.

  Let us assume that on $M$ there exist $r$ one-forms $\lambda^i$, $i=1,2,\dots,r$, such that
 $\det (X_i\hook\lambda^j)\neq 0$ on  $M$, 
   and that 
  $\der\lambda^i=\sum_{j=1}^r a^i{}_j\pi^*(\omega^j)+\sum_{j=1}^r\mu^i{}_j\dz\lambda^j$ for all $i=1,2,\dots r$,  
   with some 1-forms $\mu^i{}_j$ and some functions $a^i{}_j$ on $M$ satisfying $\mathrm{det}(a^i{}_j)\neq 0$. Consider the corresponding rank $s$ distribution
  ${\mathcal D}=\{TM\ni X~|~ X\hook\lambda^i=0, i=1,2,\dots r\}$ 
on $M$.

Then the pair $(M,{\mathcal D})$ is called a \emph{contactification} of the pair $(N,\der {\mathcal D}^\perp)$. 
\end{definition}

\begin{definition}
  A real Lie algebra $\mathfrak{g}$ spanned over $\bbR$ by the vector fields $Y$ on $M$ of the contactification $(M,{\mathcal D})$ satisfying
  \be{\mathcal L}_Y\lambda^i\dz\lambda^1\dz\dots\dz\lambda^r=0, \quad\forall i=1,2,\dots,r\label{ssymm}\ee
  is called the Lie algebra of infinitesimal symmetries of the contactification $(M,{\mathcal D})$. By definition, it is the same as the Lie algebra of infinitesimal symmetries of the distribution ${\mathcal D}$ on $M$. The vector fields $Y$ on $(M,{\mathcal D})$ satisfying \eqref{ssymm} are called infinitesimal symmetries of $(M,{\mathcal D})$, or of $\mathcal D$, for short.\label{def2}
  \end{definition}
Below, we give a nontrivial example of the notions included in Definitions \ref{def1} and \ref{def2}.
\begin{example}\label{exa4}
  Consider $N=\bbR^8$ with Cartesian coordinates $(x^1,x^2,$ $x^{3},x^{4},x^{5},x^{6},$ $x^{7},x^{8})$, and a space $\der{\mathcal D}^\perp=\Span(\omega^1,\omega^2,\omega^3,\omega^4,\omega^5,\omega^6,\omega^7)\subset\bigwedge^2N$, which is spanned by the following seven 2-forms on $N$:
  $$
  \begin{aligned}
    \omega^1=\,\,&\der x^1\dz\der x^{8}+\der x^2\dz\der x^{5}+\der x^{3}\dz\der x^{7}+\der x^{4}\dz\der x^{6}\\
    \omega^2=\,\,&-\der x^1\dz\der x^{5}+\der x^2\dz\der x^{8}+\der x^{3}\dz\der x^{6}-\der x^{4}\dz\der x^{7}\\
    \omega^3=\,\,&-\der x^1\dz\der x^{7}-\der x^2\dz\der x^{6}+\der x^{3}\dz\der x^{8}+\der x^{4}\dz\der x^{5}\\
    \omega^4=\,\,&\der x^1\dz\der x^{2}+\der x^{3}\dz\der x^{4}+\der x^{5}\dz\der x^{8}+\der x^{6}\dz\der x^{7}\\
    \omega^5=\,\,&-\der x^1\dz\der x^{6}+\der x^2\dz\der x^{7}-\der x^{3}\dz\der x^{5}+\der x^{4}\dz\der x^{8}\\
    \omega^6=\,\,&\der x^1\dz\der x^{4}+\der x^2\dz\der x^{3}-\der x^{5}\dz\der x^{7}+\der x^{6}\dz\der x^{8}\\
    \omega^7=\,\,&\der x^1\dz\der x^{3}-\der x^2\dz\der x^{4}+\der x^{5}\dz\der x^{6}+\der x^{7}\dz\der x^{8}.
  \end{aligned}
$$
  As the bundle $N$ take $M=\bbR^{7}\times\bbR^8\to N$ with coordinates $(x^1,\dots,x^8,x^9\dots,x^{15})$,  and take seven 1-forms
  $$\begin{aligned}
    \lambda^1=\,\,&\der x^9+ x^1\der x^{8}+ x^2\der x^{5}+ x^{3}\der x^{7}+ x^{4}\der x^{6}\\
    \lambda^2=\,\,&\der x^{10} - x^1\der x^{5}+ x^2\der x^{8}+ x^{3}\der x^{6}- x^{4}\der x^{7}\\
    \lambda^3=\,\,&\der x^{11} - x^1\der x^{7}- x^2\der x^{6}+ x^{3}\der x^{8}+ x^{4}\der x^{5}\\
    \lambda^4=\,\,&\der x^{12} + x^1\der x^{2}+ x^{3}\der x^{4}+ x^{5}\der x^{8}+ x^{6}\der x^{7})\\
    \lambda^5=\,\,&\der x^{13}- x^1\der x^{6}+ x^2\der x^{7}- x^{3}\der x^{5}+ x^{4}\der x^{8}\\
    \lambda^6=\,\,&\der x^{14}+ x^1\der x^{4}+ x^2\der x^{3}- x^{5}\der x^{7}+ x^{6}\der x^{8}\\
    \lambda^7=\,\,&\der x^{15}+ x^1\der x^{3}- x^2\der x^{4}+ x^{5}\der x^{6}+ x^{7}\der x^{8}.
  \end{aligned}
  $$
  This defines a rank 8 distribution ${\mathcal D}=\{TM\ni X~|~ X\hook\lambda^i=0, i=1,2,\dots 7\},$ on $M$.  
The pair $\big(M,{\mathcal D}\big)$ \emph{is a contactification of} $(N,d{\mathcal D}^\perp)$, since
  $X_i=\partial_{i+8}$, $\det(X_i\hook\lambda^j)=1$, and $\der\lambda^i=\omega^i$ for all $i=1,\dots,7$.
  In particular, in this example the rank 8 distribution
    ${\mathcal D}$ gives a \emph{2-step filtration} ${\mathcal D}_{-1}\subset\mathcal{D}_{-2}=\mathrm{T}M$, where ${\mathcal D}_{-1}={\mathcal D}$ and $\mathcal{D}_{-2}=[{\mathcal D}_{-1},{\mathcal D}_{-1}]=\mathrm{T}M$.  
\end{example}

This example is essentially taken from \`Elie Cartan's PhD thesis \cite{CartanPhd}, actually its German version. We took it as our example inspired by the following quote from Sigurdur Helgason \cite{He}: 
\begin{quote}
  Cartan represented [the simple exceptional Lie group] ${\bf F}_4$ (...) by the Pfaffian system in $\bbR^{15}$ (...).  Similar results for ${\bf E}_6$ in $\bbR^{16}$, ${\bf E}_7$ in $\bbR^{27}$ and ${\bf E}_8$ in $\bbR^{29}$ are indicated in \cite{CartanPhd}. Unfortunately, detailed proofs of these remarkable representations of the exceptional groups do not seem to be available.
\end{quote}

The 15-dimensional contactification $(M,{\mathcal D})$ from our Example \ref{exa4} is obtained in terms of the seven 1-forms $\lambda^i$, which are equivalent to the seven forms from the Cartan Pfaffian system in dimension 15 mentioned by Helgason. In particular, it follows that the \emph{distribution structure} $(M,{\mathcal D})$ has the simple exceptional Lie group ${\bf F}_4$, actually its \emph{real form} $F_I$ in the terminology of \cite{CS}, as a \emph{group of authomorphism}.

In this paper we will explain how one gets this realization of the exceptional Lie group ${\bf F}_4$, a realization of its real form $F_{II}$, and realizations of the two (out of 5) real forms $E_I$ and $E_{IV}$ of the complex simple exceptional Lie group ${\bf E}_6$. For this explanation we need some preparations consisting of recalling few notions associated with vector distributions on manifolds and spinorial representations of the orthogonal groups in space of \emph{real} spinors.

Finally we note that our approach in this paper is \emph{purely utilitarian}. We answer the question: \emph{How to get the explicit formulas in Cartesian coordinates for Pfaffian forms $(\lambda^1,\dots,\lambda^r)$, which have simple Lie algebras as symmetries?} One can study more general problems related to this on purely Lie theoretical ground. For example, one can ask when a 2-step graded nilpotent Lie algebra $\mathfrak{n}_{\minu}=\mathfrak{n}_{-2}\oplus\mathfrak{n}_{-1}$ has a given Lie algebra $\mathfrak{n}_{00}$ as a part of its Lie algebra $\mathfrak{n}_0$ of derivations preserving the strata, or a question as to  when the Tanaka prolongation of such $\mathfrak{n}_{\minu}$ with $\mathfrak{n}_{00}\subset\mathfrak{n}_0$ is finite, or simple. This is beyond the scope of our paper. A reader interested in such problems may consult e.g. \cite{AC,Alt,Krug}.
\section{Magical equation for a contactification}
The purpose of this section is to prove the following crucial lemma, about a certain algebraic equation, which we call a \emph{magical equation}. It is the boxed equation \eqref{maga} below. 

\begin{lemma}\label{l21}
  Let $(\mathfrak{n}_{00},[\cdot,\cdot]_0)$ be a finite dimensional Lie algebra, and let $\rho:\mathfrak{n}_{00}\stackrel{\mathrm{hom}}{\to} \mathrm{End}(S)$ be its finite dimensional representation in a real vector space $S$ of dimension $s$.
In addition, let $R$ be an $r$-dimensional real vector space, and $\tau:\mathfrak{n}_{00}\to \mathrm{End}(R)$, be a linear map.
Finally let $\omega$ be a linear map $\omega:\bigwedge^2S\to R$, or what is the same, let $\omega\in\mathrm{Hom}(\bigwedge^2S,R)$.

  Suppose now that the triple $(\rho,\omega,\tau)$ satisfy the following equation:
  \be\boxed{
  \omega\big(\rho(A)X,Y\big)+\omega\big(X,\rho(A)Y\big)=\tau(A)\,\omega(X,Y),}
\label{maga}
  \ee
  for all $A\in\mathfrak{n}_{00}$ and all $X,Y\in S$.
  Then we have:
  \begin{enumerate}
  \item The map $\tau$ satisfies $$\big(\,\,\tau([A,B]_0)-[\tau(A),\tau(B)]_{\mathrm{End}(R)}\,\,\big)\omega\,\,=\,\,0\quad\quad \forall\,\,A,B\in\mathfrak{n}_{00}.$$ 
  \item If the map $\tau:\mathfrak{n}_{00}\to \mathrm{End}(R)$ is a representation of $\mathfrak{n}_{00}$, i.e. if $$\tau([A,B]_0)=[\tau(A),\tau(B)]_{\mathrm{End}(R)},$$ then the real vector space $\mathfrak{g}_0:=R\oplus S\oplus\mathfrak{n}_{00}$ is a \emph{graded} Lie algebra
$$\mathfrak{g}_0=\mathfrak{n}_{-2}\oplus\mathfrak{n}_{-1}\oplus\mathfrak{n}_{00},$$
    with the graded components $$\mathfrak{n}_{-2}=R,\quad \mathfrak{n}_{-1}=S,\quad\mathrm{with}\,\, \mathfrak{n}_{00}\,\,\mathrm{as\,\,the}\,\,0\,\,\mathrm{grade},$$ and with the Lie bracket $[\cdot,\cdot]$ given by:\label{ca2}
          \begin{enumerate}
          \item if $X,Y\in \mathfrak{n}_{00}$ then $[X,Y]=[X,Y]_0$,
          \item if $A\in \mathfrak{n}_{00}$, $X\in \mathfrak{n}_{-1}$ then $[A,X]=\rho(A)X$,
          \item if $A\in \mathfrak{n}_{00}$, $X\in \mathfrak{n}_{-2}$ then $[A,X]=\tau(A)X$,
          \item $[\mathfrak{n}_{-1},\mathfrak{n}_{-2}]=[\mathfrak{n}_{-2},\mathfrak{n}_{-2}]=\{0\}$,
          \item and, if $X,Y\in \mathfrak{n}_{-1}$ then $[X,Y]=\omega(X,Y)$. 
          \end{enumerate}
        \item Moreover, in the case {\rm \eqref{ca2}} the Lie subalgebra 
          $$\mathfrak{n}_\minu=\mathfrak{n}_{-2}\oplus\mathfrak{n}_{-1}$$ of $\mathfrak{g}_0$ 
          is a 2-step graded Lie algebra, and the algebra $\mathfrak{n}_{00}$ is a Lie subalgebra of the Lie algebra $$\mathfrak{n}_0=\big\{\,\mathrm{Der}(\mathfrak{n}_\minu)\ni D\,\,\mathrm{s.t.}\,\,D\mathfrak{n}_j\subset\mathfrak{n}_j\,\,\mathrm{for}\,\, j=-1,-2\,\big\}$$ of all derivations of $\mathfrak{n}_\minu$ preserving its strata $\mathfrak{n}_{-1}$ and $\mathfrak{n}_{-2}$.
            \end{enumerate}
\end{lemma}
\begin{remark}
Note that, in the respective bases $\{ f_\mu\}_{\mu=1}^s$ in $S$ and $\{e_i\}_{i=1}^r$ in $R$, the equation \eqref{maga} is:
  \be\boxed{
  \rho(A)^\alpha{}_\mu\,\,\omega^i{}_{\alpha\nu}+\rho(A)^\alpha{}_\nu\,\,\omega^i{}_{\mu\alpha}\,\,=\,\,\tau(A)^i{}_j\,\,\omega^i{}_{\mu\nu}}\label{magb}\ee
  for all $A\in\mathfrak{n}_{00}$, all $i=1,2,\dots,r$ and all $\mu,\nu=1,2,\dots,s$.
  In this basis the condition (1) is $$\big(\,\,\tau([A,B]_0)-[\tau(A),\tau(B)]_{\mathrm{End}(R)}\,\,\big)^i{}_j\,\,\omega^j{}_{\mu\nu}\,\,=\,\,0$$ for all $i=1,2,\dots,r,\,\mu,\nu=1,2,\dots s$,  and $A,B\in\mathfrak{n}_{00}$.
  \end{remark}

\noindent
\emph{Proof of the lemma.}
The proof of part (1) is a pure calculation using the  equation \eqref{maga}. We first rewrite it in the shorthand notation as:
$$
\rho(A)\omega+\omega\rho(A)^T=\tau(A)\omega, \quad \forall  A\in\mathfrak{n}_{00}.$$
Then we have:
  $$\begin{aligned}
    \tau([A,B]_0)\omega=&\rho([A,B]_0)\omega+\omega\rho([A,B]_0)^T=\\
    &\rho(A)\rho(B)\omega-\rho(B)\rho(A)\omega+\omega\rho(B)^T\rho(A)^T-\omega\rho(A)^T\rho(B)^T=\\
    &\rho(A)\Big(\tau(B)\omega-\omega\rho(B)^T\Big)-\rho(B)\Big(\tau(A)\omega-\omega\rho(A)^T\Big)+\\&\Big(\tau(B)\omega-\rho(B)\omega\Big)\rho(A)^T-\Big(\tau(A)\omega-\rho(A)\omega\Big)\rho(B)^T=\\
    &\rho(A)\Big(\tau(B)\omega\Big)-\rho(B)\Big(\tau(A)\omega\Big)+\Big(\tau(B)\omega\Big)\rho(A)^T-\Big(\tau(A)\omega\Big)\rho(B)^T=\\
    &\tau(A)\tau(B)\omega-\tau(B)\omega\rho(A)^T-\Big(\tau(B)\tau(A)\omega-\tau(A)\omega\rho(B)^T\Big)+\\
    &\tau(B)\omega\rho(A)^T-\tau(A)\omega\rho(B)^T=\tau(A)\tau(B)\omega-\tau(B)\tau(A)\omega=\\
    &([\tau(A_,\tau(B)]_{\mathrm{End}(R)})\omega,
       \end{aligned}
$$
which proves part (1).

  The proof of parts (2) and (3) is as follows:\\
  We need to check the Jacobi identity for the bracket $[\cdot,\cdot]$.

  We first consider the representation
  $$\sigma=\tau\oplus\rho\quad\mathrm{ of} \quad\mathfrak{n}_{00}\quad \mathrm{in}\quad \mathfrak{n}_\minu=\mathfrak{n}_{-2}\oplus\mathfrak{n}_{-1},$$
  defined by
  $$\sigma(A)(Y\oplus X)=\tau(A)Y\oplus\rho(A)X, \quad \forall A\in\mathfrak{n}_{00},\,\, X\in\mathfrak{n}_{-1},\,\, Y\in\mathfrak{n}_{-2}.$$
We then prove that the representation $\sigma$ is a \emph{strata preserving derivation} in $\mathfrak{n}_\minu$. This is implied by the definitions (a)-(e) of the bracket, and the fundamental equation \eqref{maga} as follows:

The strata preserving property of $\sigma$, $\sigma(\mathfrak{n}_{-i})\subset\mathfrak{n}_{-i}$, $i=1,2$, is obvious by the definitions of $\rho$ and $\tau$. However, we need to check that $\sigma$ is a derivation, i.e. that
  \be\sigma (A)[X,Y] =[\sigma(A)X,Y]+[X,\sigma(A)Y]\label{der0}\ee
  for all $A\in\mathfrak{n}_{00}$ and for all $X,Y\in \mathfrak{n}_\minu$.
Because of the strata preserving property of $\sigma$, which we have just established, and because of the point (d) of the definition of the bracket, the equation \eqref{der0} is satisfied when both $X$ and $Y$ are in $\mathfrak{n}_{-2}$, or when $X$ is in $\mathfrak{n}_{-1}$ and $Y$ is $\mathfrak{n}_{-2}$. The only thing to be checked is if \eqref{der0} is also valid when both $X$ and $Y$ belong to $\mathfrak{n}_{-1}$.  
  But this just follows directly from \eqref{maga}, since if $ X,Y\in\mathfrak{n}_{-1}$ then
  $$\begin{aligned}\sigma(A)[X,Y]&=\sigma(A)\omega(X,Y)=\tau(A)\omega(X,Y)=\\
    &\omega(\rho(A)X,Y)+\omega(X,\rho(A)Y)=[\rho(A)X,Y]+[X,\rho(A)Y]=\\
    &[\sigma(A)X,Y]+[X,\sigma(A)Y],\quad \forall A\in\mathfrak{n}_{00}.\end{aligned}$$

  Now we return to checking the Jacobi identity for the bracket $[\cdot,\cdot]$ in $\mathfrak{g}_0$:

  On elements of the form $A,B\in \mathfrak{n}_{00}$, $Z\in \mathfrak{n}_\minu$, by (b)-(c), we have
    $$[[A,B],Z]+[[Z,A],B]+[[B,Z],A]=\Big(\sigma([A,B])-[\sigma(A),\sigma(B)]\Big)Z,$$
    which vanishes due to the representation property of $\sigma$. On the other hand, on elements $A\in \mathfrak{n}_{00}$ and $Z_1,Z_2\in \mathfrak{n}_\minu$ we have
    $$[[A,Z_1],Z_2]+[[Z_2,A],Z_1]+[[Z_1,Z_2],A]=[\sigma(A)Z_1,Z_2]+[Z_1,\sigma(A)Z_2]-\sigma(A)[Z_1,Z_2],$$
    which is again zero, on the ground of the derivation property \ref{der0} of $\sigma$. Obviously the bracket satisfies the Jacobi identity when it is restricted to $\mathfrak{n}_{00}$; it is the Lie bracket $[\cdot,\cdot]_o$ of the Lie algebra $\mathfrak{n}_{00}$. Finally, property (2) implies that $[[Z_1,Z_2],Z_3]=0$ for all $Z_1,Z_2,Z_3$ in $\mathfrak{n}_\minu$, hence the Jacobi identity is trivially satisfied for $[\cdot,\cdot]$, when it is restricted to $\mathfrak{n}_\minu$.  \hspace{9.cm}$\Box$

    \vspace{0.5cm}
    In the following we will use the map $\omega\in\mathrm{Hom}(\bigwedge^2S,R)$ satisfying the \emph{magical equation} \eqref{maga},  to construct contactifications with nontrivial symmetry algebras $\mathfrak{g}$. The setting will include  Cartan's contactification with symmetry ${\bf F}_4$ mentioned in the Helgason's quote. For this, however we need few preparations.

\section{Two-step filtered manifolds}
A \emph{2-step filtered structure} on an $(s+r)$-dimensional manifold $M$ is a pair $(M,{\mathcal D})$, in which $\mathcal D$ is a vector distribution  of rank $s$ on $M$, such that it is \emph{bracket generating} in the quickest possible way. This means that its \emph{derived distribution}  ${\mathcal D}_{-2}:=[{\mathcal D}_{-1},{\mathcal D}_{-1}]$, with ${\mathcal D}_{-1}={\mathcal D}$, is such that $${\mathcal D}_{-2}=\mathrm{T}M.$$
It provides the simplest nontrivial \emph{filtration}
$$\mathrm{T}M={\mathcal D}_{-2}\supset{\mathcal D}_{-1}$$
  of the tangent bundle $\mathrm{T}M$.

  A (local) \emph{authomorphism} of a 2-step filtered manifold $(M,{\mathcal D})$ is a (local) diffeomorphism $\phi:M\to M$ such that $\phi_*{\mathcal D}\subset{\mathcal D}$. Since authomorphism can be composed and have inverses, they form a \emph{group} $G$ of (local) authomorphism of $(M,{\mathcal D})$, also called a \emph{group of (local) symmetries of} $\mathcal D$. Infinitesimally the Lie group of authomorphism defines the \emph{Lie algebra} $\mathfrak{aut}({\mathcal D})$ \emph{of symmetries}, which is the real span of all vector fields $X$ on $M$ such that $[X,Y]\subset {\mathcal D}$ for all $Y\in{\mathcal D}$.

  Among all the 2-step filtered manifolds $(M,{\mathcal D})$ particularly simple are those which can be realized on a group manifold of a \emph{2-step nilpotent} Lie group. These are related to the notion of the \emph{nilpotent approximation} of a pair $(M,{\mathcal D})$. This is defined as follows:

  At every point $x\in M$ equipped with a 2-step filtration ${\mathcal D}_{-2}\supset{\mathcal D}_{-1}$ we have well defined vector spaces 
  $n_{-1}(x)={\mathcal D}_{-1}(x)$ and $n_{-2}(x)={\mathcal D}_{-2}(x)/{\mathcal D}_{-1}(x)$, which define a vector space
  $$\mathfrak{n}(x)=\mathfrak{n}_{-2}(x)\oplus\mathfrak{n}_{-1}(x).$$
  This vector space is naturally a \emph{Lie algebra}, with a \emph{Lie bracket} induced form the Lie bracket of vector fields in $\mathrm{T}M$. Due to the 2-step property of the filtration defined by $\mathcal D$ this Lie algebra is \emph{2-step nilpotent},
  $$[\mathfrak{n}_{-1}(x),\mathfrak{n}_{-1}(x)]=\mathfrak{n}_{-2}(x)\quad\&\quad[\mathfrak{n}_{-1}(x),\mathfrak{n}_{-2}(x)]=\{0\}.$$
  This 2-step nilpotent Lie algebra is a \emph{local invariant} of the structure $(M,{\mathcal D})$, and it is called a \emph{nilpotent approximation of} the structure $(M,{\mathcal D})$ at $x\in M$.

  This enables for defining a class of particularly simple examples of 2-step filtered structures:

  Consider a \emph{2-step nilpotent Lie algebra} $\mathfrak{n}=\mathfrak{n}_{-2}\oplus\mathfrak{n}_{-1}$, and let $M$ be a Lie \emph{group}, whose Lie algebra is $\mathfrak{n}$. The Lie algebra $\mathfrak{n}_M$ of left invariant vector fields on $M$ is isomorphic to $\mathfrak{n}$ and mirrors its gradation, $\mathfrak{n}_M={\mathfrak{n}_M}_{-2}\oplus{\mathfrak{n}_M}_{-1}$. Now, taking all linear combinations with \emph{smooth functions} coefficients of all vector fields from the graded component ${\mathfrak{n}_M}_{-1}$ of $\mathfrak{n}_M$, one defines a \emph{vector distribution} ${\mathcal D}=\Span_{{\mathcal{F}}(M)}(\mathfrak{n}_M)$ on $M$. The so constructed filtered structure $(M,{\mathcal D})$ is obviously 2-step graded and is the \emph{simplest} filtered structure with nilpotent approximation being equal to $\mathfrak{n}$ everywhere. We call this  $(M,{\mathcal D})$ structure the \emph{flat model} for all the 2-step filtered structures having the same constant nilpotent approximation $\mathfrak{n}$. 

  It is remarkable that the largest possible symmetry of all 2-step filtered structures $(M,{\mathcal D})$ is precisely the symmetry of the flat model. As such it is \emph{algebraically} determined by the nilpotent approximation $\mathfrak{n}$. This is the result of Noboru Tanaka \cite{tanaka}. To describe it we recall the notion of \emph{Tanaka prolongation}.

\begin{definition}  
  The \emph{Tanaka prolongation} of a 2-step nilpotent Lie algebra $\mathfrak{n}$ is a graded Lie algebra $\mathfrak{g}(\mathfrak{n})$ given by a direct sum
    \begin{equation}
      \mathfrak{g}(\mathfrak{n})=\mathfrak{n}\oplus\mathfrak{n}_0\oplus\mathfrak{n}_1\oplus\dots\oplus\mathfrak{n}_j\oplus\cdots,\label{gt1}\end{equation} with
      \begin{equation}\mathfrak{n}_k=\Big\{\bigoplus_{j<0}\mathfrak{n}_{k+j}\otimes\mathfrak{n}_j^*\ni A\,\,\mathrm{s.t.}\,\,A[X,Y]=[AX,Y]+[X,AY]\Big\}\label{gt}\end{equation}
      for each $k\geq 0$.
      
      Furthermore, for each $j\geq 0$, the Lie algebra $$\mathfrak{g}_j(\mathfrak{n})=\mathfrak{n}\oplus\mathfrak{n}_0\oplus\mathfrak{n}_1\oplus\dots\oplus\mathfrak{n}_j$$
      is called the Tanaka prolongations of $\mathfrak{n}$ up to $j^{th}$ order. 
      \end{definition}
Setting $[A,X]=AX$ for all $A\in \mathfrak{n}_k$ with $k\geq 0$ and for all $X\in\mathfrak{n}$ makes the condition in \eqref{gt} into the Jacobi identity. Moreover, if $A\in \mathfrak{n}_k$ and $B\in\mathfrak{n}_l$, $k,l\geq 0$, then their commutator $[A,B]\in\mathfrak{n}_{k+l}$ is defined on elements $X\in\mathfrak{n}$ inductively, according to the Jacobi identity. By this we mean that it should satisfy  
    $$[A,B]X=[A,BX]-[B,AX],$$
    which is sufficient enough to define $[A,B]$. 
    \begin{remark}
      Note, in particular, that $\mathfrak{n}_0$ is the Lie algebra of \emph{all derivations of} $\mathfrak{n}$ preserving the two strata $\mathfrak{n}_{-1}$ and $\mathfrak{n}_{-2}$ of the direct sum $\mathfrak{n}=\mathfrak{n}_{-2}\oplus\mathfrak{n}_{-1}$:
$$\mathfrak{n}_0=\big\{\,\mathrm{Der}(\mathfrak{n})\ni D\,\,\mathrm{s.t.}\,\,D\mathfrak{n}_j\subset\mathfrak{n}_j\,\,\mathrm{for}\,\, j=-1,-2\,\big\}.$$
      \end{remark}

   Although the Tanaka prolongation of a nilpotent Lie algebra $\mathfrak{n}$ is in general infinite, in this paper we will be interested in \emph{situations when the Tanaka prolongation}
$$\mathfrak{g}=\mathfrak{g}(\mathfrak{n})$$
of the $2$-step nilpotent part $$\mathfrak{n}=\mathfrak{n}_{-2}\oplus\mathfrak{n}_{-1}$$ is \emph{finite} and \emph{symmetric}, in the sense   $$\mathfrak{g}(\mathfrak{n})=\mathfrak{n}_{-2}\oplus\mathfrak{n}_{-1}\oplus\mathfrak{n}_0\oplus\mathfrak{n}_1\oplus\mathfrak{n}_2,$$
with
$$\dim(\mathfrak{n}_{-k})=\dim(\mathfrak{n}_k), \quad k=1,2.$$
Such situations \emph{are possible}, and in them the so defined Lie algebra $\mathfrak{g}(\mathfrak{n})$ is \emph{simple}. In such case the Tanaka prolongation $\mathfrak{g}(\mathfrak{n})$ is \emph{graded}, and the subalgebra $$\mathfrak{p}=\mathfrak{n}_0\oplus\mathfrak{n}_1\oplus\mathfrak{n}_2,$$ in such $\mathfrak{g}(\mathfrak{n})$ is \emph{parabolic}. Moreover, the Lie algebra 
$$\mathfrak{p}_{opp}=\mathfrak{n}_{-2}\oplus\mathfrak{n}_{-1}\oplus\mathfrak{n}_0,$$
is also a parabolic subalgebra of this simple $\mathfrak{g}(\mathfrak{n})$. It is isomorphic to $\mathfrak{p}$, $\mathfrak{p}\simeq\mathfrak{p}_{opp}$. 

Regardless of the fact if $\mathfrak{g}(\mathfrak{n})$ is finite or not, we have the following general theorem, which is a specialization of a remarkable theorem by Noboru Tanaka \cite{tanaka}:
\begin{theorem}\label{tansym}
  Consider 2-step filtered structures $(M,{\mathcal D})$, with distributions ${\mathcal D}$ having the same constant milpotent approximation  $\mathfrak{n}$. Then
  \begin{itemize}
  \item The most symmetric of all of these distribution structures is the flat model $(M,{\mathcal D})$, with $M$ being a nilpotent Lie group associated of the nilpotent approximation algebra $\mathfrak{n}$, and with $\mathcal D$ being the first component ${\mathcal D}^{-1}$ of the natural filtration on $M$ associated to the $2$-step grading in $\mathfrak{n}$.
  \item The Lie algebra of authomorphism $\mathfrak{aut}({\mathcal D})$ of the flat model structure is isomorphic to the Tanaka prolongation $\mathfrak{g}(\mathfrak{n})$ of the nilpotent approximation $\mathfrak{n}$,
$\mathfrak{aut}({\mathcal D})\simeq \mathfrak{g}(\mathfrak{n}).$   
    \end{itemize}
\end{theorem}
\begin{remark}
  This theorem is of fundamental importance for explanation of the Cartan's result about a realization of ${\bf F}_4$ in $\bbR^{15}$. As we will see Cartan's $\bbR^{15}$ is actually a \emph{domain of a chart} $({\mathcal U},\varphi)$ on a certain 2-step nilpotent Lie group $M$, with a 2-step nilpotent Lie algebra $\mathfrak{n}$,  and the equivalent description of ${\bf F}_4$ in terms of a symmetry group of the contactification $(M,{\mathcal D})$ from our Example \ref{exa4} is valid because this contactification is just the flat model for the 2-step filtration $(M,{\mathcal D})$ with the nilpotent approximation $\mathfrak{n}$. 
\end{remark}

Using the information about the Tanaka prolongation of a nilpotent Lie algebra $\mathfrak{n}$ we can enlarge our Lemma \ref{l21} by changing its point (3) into the following more complete form:
\begin{lemma}\label{l213}
  With all the assumptions of Lemma \ref{l21}, and with points {\rm(1)} and {\rm (2)} as in Lemma \ref{l21}, its point {\rm(3)} is equivalent to 
    \begin{itemize}
    \item[]  {\rm(3)} Moreover, in the case {\rm \eqref{ca2}} the Lie subalgebra 
          $$\mathfrak{n}_\minu=\mathfrak{n}_{-2}\oplus\mathfrak{n}_{-1}$$ of $$\mathfrak{g}_0=\mathfrak{n}_{-2}\oplus\mathfrak{n}_{-1}\oplus \mathfrak{n}_{00}$$ 
    is a 2-step graded nilpotent Lie algebra, and the algebra $\mathfrak{n}_{00}$ is a Lie \emph{subalgebra} of the Tanaka prolongation up to $0^{th}$ order $\mathfrak{g}_0(\mathfrak{n}_\minu)$ of the Lie algebra  $\mathfrak{n}_\minu=\mathfrak{n}_{-2}\oplus\mathfrak{n}_{-1}$.
  \end{itemize}
 \end{lemma}
 \begin{remark} The term `... $\mathfrak{n}_{00}$ is a Lie \emph{subalgebra} of the Tanaka prolongation up to $0^{th}$ order $\mathfrak{g}_0(\mathfrak{n}_\minu)$ of the Lie algebra  $\mathfrak{n}_\minu=\mathfrak{n}_{-2}\oplus\mathfrak{n}_{-1}$..' in the above lemma, means that $\mathfrak{n}_{00}$, although nontrivial, is in general only a subalgebra of the $$\mathfrak{n}_0=\big\{\,\mathrm{Der}(\mathfrak{n}_\minu)\ni D\,\,\mathrm{s.t.}\,\,D \mathfrak{n}_j\subset \mathfrak{n}_j\,\,\mathrm{for}\,\, j=-1,-2\,\big\},\quad\quad \mathfrak{n}_{00}\subsetneq \mathfrak{n}_0,$$ which is the \emph{full} $0$ graded component of the Tanaka prolongation of $\mathfrak{n}_\minu$. So for applications it is reasonable to choose $\mathfrak{n}_{00}$ as large as possible.
   \end{remark}

 \section{Construction of contactifications with nice symmetries}

 Consider a Lie algebra $(\mathfrak{n}_{00},[\cdot,\cdot]_0)$ and its two real representations $(\rho,S)$, $(\tau,R)$, in the respective real $s$- and $r$-dimensional vectors spaces $S$ and $R$.
 Let $S=\bbR^s$, $R=\bbR^r$, and let $\{ f_\mu\}_{\mu=1}^s$ and $\{e_i\}_{i=1}^r$ be respective bases in $S$ and in $R$. Let $\{f^\mu\}_{\mu=1}^s$ be a basis in the vector space $S^*$ dual to the basis $\{ f_\mu\}_{\mu=1}^s$ , $f_\nu\hook f^\mu=\delta_\nu{}^\mu$.
To be in a situation of Lemma \ref{l21} we also assume that we have the homomorphism $\omega\in\mathrm{Hom}\bigwedge^2 S,R)$ satisfying the magical equation \eqref{maga}. 

 Then the map $\omega$ is $$\omega=\tfrac12\omega^i{}_{\mu\nu}e_i\otimes f^\mu\dz f^\nu,$$ and it defines the coefficients $\omega^i{}_{\mu\nu}$, $i=1,\dots,r$, $\mu,\nu=1,2,\dots s$, which satisfy $\omega^i{}_{\mu\nu}=-\omega^i{}_{\nu\mu}$.

 Now, consider an $s$-dimensional manifold, which is an open set $N$ of $\bbR^s$,  $N\subset\bbR^s$, with coordinates $(x^\mu)_{\mu=1}^r$. Then, we have $r$ two-forms $(\omega^i)_{i=1}^r$ on $N$ defined by
 $$\omega^i=\tfrac12\omega^i{}_{\mu\nu}\der x^\mu\dz\der x^\nu.$$
 This produces an $(N,\der{\mathcal D}^\perp)$ structure on $N$, with
 $$\der{\mathcal D}^\perp=\Span_\bbR(\omega^1,\dots,\omega^r).$$
We contactify it. For this we take a local $M=\bbR^r\times N$, with coordinates $\big(u^i,x^\mu\big)(_{i=1}^r)(_{\mu=1}^s)$, and define the `contact forms' on $M$ by
 $$\lambda^i=\der u^i+\omega^i{}_{\mu\nu}x^\mu\der x^\nu.$$
 Because of Lemmas \ref{l21} and \ref{l213} the distribution ${\mathcal D}$ on $M$ defined by this contactification as in Definition \ref{def1}, equips $M$ with a \emph{2-step} filtered structure having ${\mathcal D}_{-1}=\mathcal D$. This has rank $s$. Now using Lemmas \ref{l21} and \ref{l213}, and Tanaka's Theorem \ref{tansym}, we get the following corollary.
 \begin{corollary}
   Let $M=\bbR^r\times\bbR^s$ and let $$\lambda^i=\der u^i+\omega^i{}_{\mu\nu}x^\mu\der x^\nu, \quad i=1,\dots r,$$
   with $\omega$ being a solution of the magical equation \ref{maga} such that $\mathrm{Im}(\omega)=R$.
   Consider the distribution structure $(M,{\mathcal D})$ with 
 a rank $r$ distribution $${\mathcal D}=\{\mathrm{T}M\ni X,\,\, s.t.\,\,X\hook\lambda^i=0,\,\,i=1,\dots,r\}$$
on $M$.    
   Then, the Lie algebra of authomorphism $\mathfrak{aut}({\mathcal D})$ of $(M,{\mathcal D})$ is isomorphic to the Tanaka prolongation of the 2-step nilpotent Lie algebra $\mathfrak{n}_\minu=R\oplus S$ defined in point {\rm (3)} of Lemma \ref{l21} or \ref{l213}. The Lie algebra $\mathfrak{g}_0=R\oplus S\oplus \mathfrak{n}_{00}$ is nontrivially contained in the Tanaka prolongation up to the $0^{th}$ order $\mathfrak{g}_0(\mathfrak{n}_\minu)$ of $\mathfrak{n}_\minu$, with $\{0\}\neq \mathfrak{n}_{00}\subset\mathfrak{n}_0$, and as such is a subalgebra of the algebra of $\mathfrak{aut}({\mathcal D})$.   \label{cruco}
   \end{corollary}
 \section{Majorana spinor representations of $\soa(p,q)$}\label{spintraut}
 In this section we will explain how to construct the real spin representations of the Lie algebras $\soa(p,q)$, in cases when $p=n$, $q=n-1$, or $p=q=n$, $n=1,2,\dots n$. We will also give a construction of these representations for $\soa(0,n)$. We emphasize that we are only interested in \emph{real} spin representations. They share a general name of \emph{Majorana representations}. Our presentation of this material is adapted from \cite{traut}.

 We will need Pauli matrices
 \be
\sigma_x=\bma 0&1\\1&0\ema,\quad \epsilon=-i\sigma_y=\bma 0&-1\\1&0\ema,\quad \sigma_z=\bma 1&0\\0&-1\ema,\label{pauu}
\ee
and  the $2\times 2$ identity matrix
\be
I=\bma 1&0\\0&1\ema.\label{pauu1}\ee
We have the following identities:
\be\begin{aligned}
  &\sigma_x^2=\sigma_z^2=-\epsilon^2=I\\
  &\sigma_x\epsilon=-\epsilon\sigma_x=\sigma_z,\quad\sigma_z\sigma_x=-\sigma_x\sigma_z=-\epsilon,\quad \epsilon\sigma_z=-\sigma_z\epsilon=\sigma_x.
  \end{aligned}\label{iden}
\ee

Now we quote \cite{traut}:
\begin{quote}
  With this notation, \emph{restricting to low dimensions} $p+q=4,5,6$ \emph{and} $7$, the real representations of the Clifford algebra ${\mathcal C}\ell(0,p+q)$ are all in dimension $s=8$, and are generated by the $p+q$ matrices $\rho_1,\dots,\rho_{(p+q)}$ given by:
\be\begin{aligned}
  \rho_1&=\sigma_z\otimes I\otimes \epsilon\\
  \rho_2&=\sigma_z\otimes \epsilon\otimes \sigma_x\\
  \rho_3&=\sigma_z\otimes \epsilon\otimes \sigma_z\\
  \rho_4&=\sigma_x\otimes \epsilon \otimes I\\
  \rho_5&=\sigma_x\otimes \sigma_x\otimes \epsilon\\
  \rho_6&=\sigma_x\otimes \sigma_z\otimes \epsilon\\
   \rho_7&=\epsilon\otimes I\otimes I.
  \end{aligned}\label{cl07}\ee
  The 8 matrices $\theta_\mu=\sigma_x\otimes\rho_\mu$, $\mu=1,\dots, 7$ and $\theta_8=\epsilon\otimes I \otimes I\otimes I$ give the real representation of ${\mathcal C}\ell(0,8)$ in $S=\bbR^{16}$. Dropping the first factor in $\rho_1,\rho_2,\rho_3$ one obtains the matrices generating a representation of ${\mathcal C}\ell(0,3)$ in $S=\bbR^4$, etc.
  \end{quote}

Majorana representations of $\soa(n-1,n)$ in dimension $s=2^{n-1}$ are called \emph{Pauli representations}, and Majorana representations of $\soa(n,n)$ in dimension $s=2^n$, are called \emph{Dirac representations}.

To construct them we need generalizations of the \emph{Pauli} $\sigma$ matrices and \emph{Dirac} $\gamma$ \emph{matrices}. The construction of those is \emph{inductive}.

It starts with $p+q=1$ with one matrix $\sigma_1=1$, and for every $n=1,2,\dots$, it alternates between $p+q=2n-1$ of Pauli matrices $\sigma_\mu$, $\mu=1,\dots, 2n-1$,  and $p+q=2n$ of Dirac matrices $\gamma_\mu$, $\mu=1,\dots, 2n$.

Again quoting Trautman \cite{traut} we have:

\begin{enumerate}
\item In dimension $p+q=1$ put $\sigma_1=1$.
\item Given  $2^{n-1}\times 2^{n-1}$ matrices $\sigma_\mu$, $\mu=1,\dots,2n-1$, define
  $$\gamma_\mu=\bma 0&\sigma_\mu\\\sigma_\mu&0\ema\,\,\mathrm{for}\,\,\mu=1,\dots,2n-1,$$
  and
  $$\gamma_{2n}=\bma 0&-I\\I&0\ema,$$
   where $I$ is the identity $2^{n-1}\times 2^{n-1}$ matrix.
\item Given $2^n\times 2^n$ matrices $\gamma_\mu$, $\mu=1,\dots,2n$, define $\sigma_\mu=\gamma_\mu$ for $\mu=1,\dots, 2n$, and $\sigma_{2n+1}=\gamma_1\dots\gamma_{2n}$, so that for $n>0$,
  $$\sigma_{2n+1}=\bma I&0\\0&-I\ema.$$
 
\end{enumerate}

In every dimension $p+q=2n-1$, $n\geq 1$, the Pauli matrices $\sigma_\mu$, $\mu=1,\dots,2n-1$ satisfy
$$\sigma_\mu\sigma_\nu+\sigma_\nu\sigma_\mu=2g_{\mu\nu}\underbrace{\big(I\otimes\dots\otimes I\big)}_{n-1\,\, \mathrm{times}},$$
where the $(2n-1)\times (2n-1)$ symmetric matrix $(g_{\mu\nu})$ is \emph{diagonal}, and has the following diagonal elements:
$$(g_{\mu\nu})=\mathrm{diag}\underbrace{(1,-1,\dots,-1,1)}_{(2n-1)\,\, \mathrm{times}}.$$
Likewise, in every dimension $p+q=2n$, $n\geq 1$, the Dirac matrices $\gamma_\mu$, $\mu=1,\dots,2n$ satisfy 
$$\gamma_\mu\gamma_\nu+\gamma_\nu\gamma_\mu=2g_{\mu\nu}\underbrace{\big(I\otimes\dots\otimes I\big)}_{n\,\, \mathrm{times}},$$
where the $(2n)\times (2n)$ symmetric matrix $(g_{\mu\nu})$ is \emph{diagonal}, and has the following diagonal elements:
$$(g_{\mu\nu})=\mathrm{diag}\underbrace{(1,-1,\dots,1,-1)}_{2n\,\, \mathrm{times}}.$$
Therefore, for each $n=1,2,\dots$ the set $\{\sigma_\mu\}_{\mu=1}^{2n-1}$ of Pauli matrices generates the elements of a real $2^{n-1}$-dimensional representation of the Clifford algebra ${\mathcal C}\ell(n-1,n)$, and the set $\{\gamma_\mu\}_{\mu=1}^{2n}$  of Dirac matrices generates the elements of a real $2^{n}$-dimensional representation of the Clifford algebra ${\mathcal C}\ell(n,n)$.

Then, in turn, these real Clifford algebras representations can be further used to define the real spin representations of the Lie algebras $\soa(p+q,0)$, $\soa(n-1,n)$ and $\soa(n,n)$ as follows. One obtains all the generators of the spin representation of $\soa(g)$ by spanning it by all the elements of the form
\begin{itemize}
\item $\tfrac12\rho_\mu\rho_\nu$, with $1\leq \mu<\nu\leq (p+q)$, in the case of $\soa(p+q,0)$, $p+q=3,5,6,7$;
\item $\tfrac12\theta_\mu\theta_\nu$, with $1\leq \mu<\nu\leq 8$, in the case of $\soa(8,0)$;
\item $\tfrac12\sigma_\mu\sigma_\nu$, with $1\leq \mu<\nu\leq (p+q)=2n-1$, in the case of $\soa(n-1,n)$;
  \item $\tfrac12\gamma_\mu\gamma_\nu$, with $1\leq \mu<\nu\leq (p+q)=2n$, in the case of $\soa(n,n)$.
  \end{itemize}
For further details consult \cite{traut}.

We will use all this information in next sections, when we create examples. 
\section{Application: Obtaining the flat model for (3,6) distributions}
Let $(\rho,S)$ be the defining representation of $\soa(3)$ in $S=\bbR^3$. It can be generated by:
\be\rho(A_1)=\bma 0&0&-1\\0&0&0\\1&0&0\ema,\quad\rho(A_2)=\bma 0&1&0\\-1&0&0\\0&0&0\ema,\quad\rho(A_1)=\bma 0&0&0\\0&0&-1\\0&1&0\ema.\label{ro3a}\ee
And let $(\tau,R)$ be an equivalent 3-dimensional representation of $\soa(3)$ given by
\be \tau(A_1)=\bma 0&0&-1\\0&0&0\\1&0&0\ema,\quad\tau(A_2)=\bma 0&-1&0\\1&0&0\\0&0&0\ema,\quad\tau(A_1)=\bma 0&0&0\\0&0&1\\0&-1&0\ema.\label{tau3a}\ee
We claim that for these two representations of $\soa(3)$, in the standard bases in $S=\bbR^3$, $R=\bbR^3$, the magical equation \eqref{magb} has the following solution:
$$\omega^1{}_{\mu\nu}=\bma 0&0&0\\0&0&-1\\0&1&0\ema,\quad\omega^2{}_{\mu\nu}=\bma 0&0&-1\\0&0&0\\1&0&0\ema,\quad\omega^3{}_{\mu\nu}=\bma 0&-1&0\\1&0&0\\0&0&0\ema.$$
Now using this solution $(\rho,\tau,\omega)$ of the magical equation \eqref{maga} we use the Corollary \ref{cruco} with $\lambda^i=\der u^u+\omega^i{}_{\mu\nu}x^\mu\der x^\nu$, and obtain the following theorem.
\begin{theorem}
  Let $M=\bbR^6$ with coordinates $(u^1,u^2,u^3,x^1,x^2,x^3)$ and consider three 1-forms
  $$\begin{aligned}
    \lambda^1=&\der u^1+x^2\der x^3\\
    \lambda^2=&\der u^2+x^1\der x^3\\
    \lambda^3=&\der u^3+x^1\der x^2
  \end{aligned}$$
  on $M$.
  Then the rank 3 distribution ${\mathcal D}$ on $M$ defined by ${\mathcal D}=\{\mathrm{T}\bbR^6\ni X\,\,|\,\, X\hook \lambda^i=0,\,\,i=1,2,3\}$ has its Lie algebra of infinitesimal symmetries $\mathfrak{aut}{\mathcal D}$ isomorphic to the Tanaka prolongation of $\mathfrak{n}_\minu=R\oplus S$ where $(\rho,S=\bbR^3)$ and $(\tau,R=\bbR^3)$ are the respective representations \eqref{ro3a}, \eqref{tau3a} of $\mathfrak{n}_{00}=\soa(3)$.

  The symmetry algebra $\mathfrak{aut}({\mathcal D})$ is isomorphic to the simple graded Lie algebra $\soa(4,3)$,
  $$\mathfrak{aut}({\mathcal D})=\soa(4,3),$$
  with the following gradation:
   $$\mathfrak{aut}({\mathcal D})=\mathfrak{n}_{-2}\oplus\mathfrak{n}_{-1}\oplus\mathfrak{n}_0\oplus\mathfrak{n}_1\oplus\mathfrak{n}_2,$$
  with $\mathfrak{n}_{-2}=R$, $\mathfrak{n}_{-1}=S$,
  $$
    \mathfrak{n}_0=\gla(3,\bbR)\supset\mathfrak{n}_{00}$$
  $\mathfrak{n}_{1}=S^*$, $\mathfrak{n}_{2}=R^*$,
 which is inherited from the distribution structure $(M,{\mathcal D})$. The duality signs $*$ at $R^*$ and $S^*$ above are with respect to the Killing form in $\soa(4,3)$.

 The contactification $(M,{\mathcal D})$ is locally a flat model for the parabolic geometry of type $\big({\bf Spin}(4,3),P\big)$ related to the following \emph{crossed} Satake diagram: \tikzset{/Dynkin diagram/fold style/.style={stealth-stealth,thin,
shorten <=1mm,shorten >=1mm}}\begin{dynkinDiagram}[edge length=.5cm]{B}{oot}
   \end{dynkinDiagram}.
\end{theorem}
\begin{proof}
  Proof is by calculating the Tanaka prolongation of $\mathfrak{n}_\minu=R\oplus S$, which is $\gla(3,\bbR)$, naturally graded by the Tanaka prolongation algebraic procedure precisely as $\mathfrak{aut}({\mathcal D})$ in the statement of the theorem. 
  \end{proof}
 \section{Application: Obtaining Biquard's 7-dimensional flat quaternionic contact manifold via contactification using spin representations of $\soa(1,2)$ and $\soa(3,0)$}
 According to Trautman's procedure \cite{traut} there is a real representation of ${\mathcal C}\ell(0,3)$ in $\bbR^4$. There also is an analogous representation of ${\mathcal C}\ell(1,2)$. Both of them are generated by the $\sigma$ matrices
 $$\sigma_1=\bma 0&-1&0&0\\1&0&0&0\\0&0&0&-1\\0&0&1&0\ema,\quad\sigma_2=\bma 0&0&0&-\varepsilon\\0&0&-\varepsilon&0\\0&1&0&0\\1&0&0&0\ema,\quad\sigma_3=\bma 0&0&-\varepsilon&0\\0&0&0&\varepsilon\\1&0&0&0\\0&-1&0&0\ema,$$
 where $$\varepsilon=1\quad \mathrm{for}\quad {\mathcal C}\ell(0,3),$$ and $$\varepsilon=-1\quad\mathrm{for}\quad{\mathcal C}\ell(2,1).$$ One can check that these matrices\footnote{In Trautman's quote in the previous section, these matrices where denoted by $\rho_1$, $\rho_2$, $\rho_3$, and they were only explicitly given for $\varepsilon=1$.} satisfy the (representation of) Clifford algebra relations:
 $$\sigma_\mu\sigma_\nu+\sigma_\nu\sigma_\mu=2g_{\mu\nu}\big(I\otimes I)$$
 with all $g_{\mu\nu}$ being zero, except $g_{11}=-1$, $g_{22}=g_{33}=-\varepsilon$.

 This leads to the following spinorial
 representation $\rho$ of $\soa(0,3)$ or $\soa(1,2)$
 \be\rho(A_1)=-\tfrac12\sigma_3,\quad \rho(A_2)=\tfrac12\sigma_2,\quad\rho(A_3)=-\tfrac12\varepsilon\sigma_1.\label{so31}\ee
 Here $(A_1,A_2,A_3)$ constitutes a basis for $\soa(0,3)$ when $\varepsilon=1$ and for $\soa(2,1)$ when $\varepsilon=-1$. This can be extended to the representation of
 $$\mathfrak{n}_{00}=\bbR\oplus \soa\big(\tfrac{1-\varepsilon}{2},\tfrac{5+\varepsilon}{2}\big)$$
 in $S=\bbR^4$ by setting the value of $\rho$ on the generator $A_4=\mathrm{Id}$ as
 \be\rho(A_4)=\tfrac12\mathfrak(I\otimes I).\label{so32}\ee
 For this representation of $\bbR\oplus \soa\big(\tfrac{1-\varepsilon}{2},\tfrac{5+\varepsilon}{2}\big)$, the magical equation \eqref{maga} has a following solution
$$\begin{aligned}
   \omega^1{}_{\mu\nu}=&\bma 0&0&1&0\\0&0&0&-1\\-1&0&0&0\\0&1&0&0\ema,\quad \omega^2{}_{\mu\nu}=\bma 0&0&0&-1\\0&0&-1&0\\0&1&0&0\\1&0&0&0\ema,\\
   &\\
  &\omega^1{}_{\mu\nu}=\bma 0&-\varepsilon&0&0\\\varepsilon&0&0&0\\0&0&0&-1\\0&0&1&0\ema,
 \end{aligned}$$
 with
 \be\begin{aligned}
   \tau(A_1)=\bma 0&0&0\\0&0&\varepsilon\\0&-1&0\ema,\quad &\tau(A_2)=\bma 0&0&-\varepsilon\\0&0&0\\-1&0&0\ema,\quad\tau(A_3)=\bma 0&-\varepsilon&0\\\varepsilon&0&0\\0&0&0\ema\\
   &\tau(A_4)=\bma 1&0&0\\0&1&0\\0&0&1\ema.
 \end{aligned}
 \label{so33}\ee
 This in particular gives the vectorial representation $\tau$ of 
 $$\mathfrak{n}_{00}=\bbR\oplus \soa\big(\tfrac{1-\varepsilon}{2},\tfrac{5+\varepsilon}{2}\big)$$
in $R=\bbR^3$.

Now, by using this solution for $(\rho,\omega,\tau)$ and applying our Corollary \ref{cruco} we have an $(s=4)$-dimensional manifold $N=\bbR^4$, equipped with $r=3$ two-forms $\omega^i=\tfrac12\omega^{}_{\mu\nu}\der x^\mu\dz\der x^\nu$, $i=1,2,3$, which contactifies to an $(s+r)=7$-dimensional manifold $M=\bbR^7$ having a distribution structure $(M,{\mathcal D})$ defined as an annihilator of the $r=3$ one-forms $\lambda^i=\der u^i+\omega^i{}_{\mu\nu}x^\mu\der x^\nu$, $i=1,2,3$.
We have the following theorem.
\begin{theorem}
  Let $M=\bbR^7$ with coordinates $(u^1,u^2,u^3,x^1,x^2,x^3,x^4)$, and consider three 1-forms $\lambda^1,\lambda^2,\lambda^3$ on $M$ given by
  $$\begin{aligned}
    \lambda^1=&\der u^1+x^1\der x^3-x^2\der x^4,\\
    \lambda^2=&\der u^2-x^1\der x^4-x^2\der x^3,\\
    \lambda^3=&\der u^3-\varepsilon x^1\der x^2-x^3\der x^4,
        \end{aligned}\quad\quad\mathrm{with}\quad\quad\varepsilon=\pm1.$$
  The rank 4 distribution ${\mathcal D}$ on $M$ defined as ${\mathcal D}=\{\mathrm{T}\bbR^7\ni X\,\,|\,\,X\hook\lambda^1=X\hook\lambda^2=X\hook\lambda^3=0\}$ has its Lie algebra of infinitesimal authomorphism $\mathfrak{aut}(\mathcal D)$ isomorphic to the Tanaka prolongation of $\mathfrak{n}_{\minu}=R\oplus S$, where $(\rho,S=\bbR^4)$ is the spinorial representation \eqref{so31}-\eqref{so32} of $\mathfrak{n}_{00}=\bbR\oplus \soa\big(\tfrac{1-\varepsilon}{2},\tfrac{5+\varepsilon}{2}\big)$, and $(\tau,R=\bbR^3)$ is the vectorial representation \eqref{so33} of $\mathfrak{n}_{00}$.

  The symmetry algebra $\mathfrak{aut}({\mathcal D})$ is isomorphic to the simple Lie algebra $\spa\big(\tfrac{1-\varepsilon}{2},\tfrac{5+\varepsilon}{2}\big)$,
  $$\mathfrak{aut}({\mathcal D})=\spa\big(\tfrac{1-\varepsilon}{2},\tfrac{5+\varepsilon}{2}\big),$$
  having the following natural gradation 
  $$\mathfrak{aut}({\mathcal D})=\mathfrak{n}_{-2}\oplus\mathfrak{n}_{-1}\oplus\mathfrak{n}_0\oplus\mathfrak{n}_1\oplus\mathfrak{n}_2,$$
  with $\mathfrak{n}_{-2}=R$, $\mathfrak{n}_{-1}=S$,
  $$\begin{aligned}
    \mathfrak{n}_0=&\mathfrak{n}_{00}\oplus \soa\big(\tfrac{1-\varepsilon}{2},\tfrac{5+\varepsilon}{2}\big)=\\&\bbR\oplus \soa\big(\tfrac{1-\varepsilon}{2},\tfrac{5+\varepsilon}{2}\big)\oplus \soa\big(\tfrac{1-\varepsilon}{2},\tfrac{5+\varepsilon}{2}\big),\end{aligned}$$
  $\mathfrak{n}_{1}=S^*$, $\mathfrak{n}_{2}=R^*$,
 which is inherited from the distribution structure $(M,{\mathcal D})$. The duality signs $*$ at $R^*$ and $S^*$ above are with respect to the Killing form in $\spa\big(\tfrac{1-\varepsilon}{2},\tfrac{5+\varepsilon}{2}\big)$.

 The contactification $(M,{\mathcal D})$ is locally a flat model for the parabolic geometry of type $\Big(\spg\big(\tfrac{1-\varepsilon}{2},\tfrac{5+\varepsilon}{2}\big),P\Big)$ related to the following \emph{crossed} satake diagrams:
 \begin{enumerate}
 \item
   \tikzset{/Dynkin diagram/fold style/.style={stealth-stealth,thin,
shorten <=1mm,shorten >=1mm}}\begin{dynkinDiagram}[edge length=.5cm]{C}{*t*}
   \end{dynkinDiagram}
     in the case of $\varepsilon=1$, and
     \item \tikzset{/Dynkin diagram/fold style/.style={stealth-stealth,thin,
shorten <=1mm,shorten >=1mm}}\begin{dynkinDiagram}[edge length=.5cm]{C}{oto}
     \end{dynkinDiagram}
       in the case of $\varepsilon=-1$.
   \end{enumerate}
\end{theorem}
\begin{remark}
  When $\varepsilon=1$ the flat parabolic geometry described in the above theorem is the lowest dimensional example of the \emph{quaternionic contact} geometry considered by Biquard \cite{bicquard}.  
  \end{remark}

 \section{Application: Obtaining the exceptionals from contactifications of spin representations; the $\mathfrak{f}_4$ case}

 We will now explain the Cartan realization of the simple exceptional Lie algebra $\mathfrak{f}_4$ in dimension $\bbR^{15}$ mentioned in the introduction.

 The Satake diagrams for the real forms of the complex simple exceptional Lie algebra $\mathfrak{f}_4$ are as follows:\\
 \centerline{\begin{dynkinDiagram}[edge length=.4cm]{F}{****}
   \end{dynkinDiagram},\hspace{0.5cm}
   \begin{dynkinDiagram}[edge length=.4cm]{F}{***o}
   \end{dynkinDiagram},\hspace{0.5cm}
   \begin{dynkinDiagram}[edge length=.4cm]{F}{oooo}\end{dynkinDiagram}.}
 The first diagram corresponds to the \emph{compact} real form of $\mathfrak{f}_4$ an is not interesting for us. The other two diagrams are interesting:
 \begin{enumerate}
 \item The last, \begin{dynkinDiagram}[edge length=.4cm]{F}{oooo}\end{dynkinDiagram}, corresponds to the \emph{split} real form $\mathfrak{f}_I$ , and
   \item the middle one, \begin{dynkinDiagram}[edge length=.4cm]{F}{***o}
   \end{dynkinDiagram}, denoted by $\mathfrak{f}_{II}$ in \cite{CS}, is also interesting, since similarly to $\mathfrak{f}_I$, it defines a \emph{parabolic geometry} in dimension 15. 
 \end{enumerate}
 Crossing the last node on the right in the diagrams for $\mathfrak{f}_I$ or $\mathfrak{f}_{II}$, as in\\
 \centerline{ \tikzset{/Dynkin diagram/fold style/.style={stealth-stealth,thin,
shorten <=1mm,shorten >=1mm}}\begin{dynkinDiagram}[edge length=.5cm]{F}{ooot}
   \end{dynkinDiagram} \hspace{0.5cm} or \hspace{0.5cm} \tikzset{/Dynkin diagram/fold style/.style={stealth-stealth,thin,
shorten <=1mm,shorten >=1mm}}\begin{dynkinDiagram}[edge length=.5cm]{F}{***t}
   \end{dynkinDiagram},} we see that in both algebras there exist \emph{parabolic subalgebras} $\mathfrak{p}_I$ or $\mathfrak{p}_{II}$, respectively, of dimension 37, $\dim(\mathfrak{p}_I)=\dim(\mathfrak{p}_{II})=37$. In both respective cases, these choices of parabolics, define similar gradations in the corresponding real forms $\mathfrak{f}_I$, $\mathfrak{f}_{II}$, of the simple exceptional Lie $\mathfrak{f}_4$:
$$
 \mathfrak{f}_A=\mathfrak{n}_{-2A}\oplus\mathfrak{n}_{-1A}\oplus\mathfrak{n}_{0A}\oplus\mathfrak{n}_{1A}\oplus\mathfrak{n}_{2A}\quad \mathrm{for}\quad A=I,II,
$$
 with
 $$\mathfrak{n}_{-A}=\mathfrak{n}_{-2A}\oplus\mathfrak{n}_{-1A}\quad \mathrm{for}\quad A=I,II,$$
 being 2-step nilpotent and having grading components $\mathfrak{n}_{-2A}$ and $\mathfrak{n}_{-1A}$ of respective dimension $r_A=7$ and $s_A=8$,
 $$r_A=\dim(\mathfrak{n}_{-2A})=7,\quad\quad s_A=\dim(\mathfrak{n}_{-1A})=8\quad \mathrm{for}\quad A=I,II.$$
 The Lie algebra $\mathfrak{n}_{0A}$ in the Tanaka prolongation of $\mathfrak{n}_{-A}$ up to $0^{th}$ order is
 \begin{enumerate}
 \item $\mathfrak{n}_{0I}=\bbR\oplus\soa(4,3)$ in the case of $\mathfrak{f}_I$, and
   \item $\mathfrak{n}_{0II}=\bbR\oplus\soa(0,7)$ in the case of $\mathfrak{f}_{II}$.
 \end{enumerate}
 Thus, from the analysis performed here, we see that there exists two different 2-step filtered structures $(M_I,{\mathcal D}_I)$ and $(M_{II},{\mathcal D}_{II})$, both in dimension 15, with the respective $F_I$-symmetric, or $F_{II}$-symmetric flat models, realized on $M_I=F_I/P_I$ or $M_{II}=F_{II}/P_{II}$. Here $F_I$ and $F_{II}$ denote the real Lie groups whose Lie algebras are $\mathfrak{f}_I$ and $\mathfrak{f}_{II}$, respectively. Similarly $P_I$ and $P_{II}$ are parabolic subgroups of respective $F_i$ and $F_{II}$, whose Lie algebras are $\mathfrak{p}_I$ and $\mathfrak{p}_{II}$. Recalling that each of the real groups $\sog(4,3)$ and $\sog(0,7)$ has \emph{two} real irreducible representations $\rho$ in dimension $s=8$ and $\tau$ in dimension $r=7$, with the 8-dimensional representation $\rho$ being the \emph{spin} representation of either $\sog(4,3)$ or $\sog(0,7)$, we can now give the explicit realizations of the ${\bf F}_4$-symmetric structures $(M_A,{\mathcal D}_A)$ for $A=I,II$.

 \subsection{Cartan's realization of $\mathfrak{f}_I$}\label{41}
 The plan is to start with the Lie algebra $\mathfrak{n}_{00}=\bbR\oplus\soa(4,3)$, as in the crossed Satake diagram  \tikzset{/Dynkin diagram/fold style/.style={stealth-stealth,thin,
shorten <=1mm,shorten >=1mm}}\begin{dynkinDiagram}[edge length=.5cm]{F}{ooot}
 \end{dynkinDiagram}
  of $\mathfrak{f}_I$, and its two representations: \begin{itemize}
 \item a representaion $(\rho,S=\bbR^8)$, corresponding to the spin representation of $\sog(4,3)$ in $(s=8)$-dimensional space $\mathfrak{n}_{-1}=S$ of real Pauli spinors, and
   \item a representation $(\tau,R=\bbR^7)$, corresponding to the vectorial representation of $\sog(4,3)$ in $(r=7)$-dimensional space $\mathfrak{n}_{-2}=R$ of vectors in $\bbR^{(4,3)}$.
 \end{itemize}
 Having these two representations of $\mathfrak{n}_{00}=\bbR\oplus\soa(4,3)$ in the same basis, we will then solve the equations \eqref{maga} for the map $\omega\in\mathrm{Hom}(\bigwedge^2S,R)$ which will give us the commutators between elements in $\mathfrak{n}_{-1}$. This via Corollary \ref{cruco} will provide the explicit realization of the 15-dimensional contactification $(M,{\mathcal D})$ with the exception simple Lie algebra $F_I$ as its symmetry.

Actually, the passage from $\rho$ to $\tau$ in the above plan, is a bit tricky, since we need to have these representations expressed in the same basis. To handle with this obstacle, we will start with the spin representation $\rho$, in the space of \emph{Pauli spinors} $S$, and then we will use the fact that the skew representation $\rho\dz\rho$ in the space of the bispinors $\bigwedge^2S$ decomposes as
 $$\textstyle \bigwedge^2S=\bigwedge_{21}\oplus\bigwedge_{7},$$ 
 where $\bigwedge_{21}$ is the 21-dimensional \emph{adjoint representation} of $\sog(4,3)$ and $\bigwedge_7$ is its 7-dimensional \emph{vectorial representation} $\tau$. In this way we will have the two representations $(\rho,S)$ and $(\tau,R=\bigwedge_7)$, expressed in the same basis $\{A_I\}$ of $\bbR\oplus\sog(4,3)$, and will apply the Corollary \ref{cruco} to get the desired $F_I$-symmetric contactification in dimension 15. On doing this we will use notation from Section \ref{spintraut}.

According to \cite{traut}, the real 8-dimensional \emph{representation of the Clifford algebra} ${\mathcal C}\ell(4,3)$ is generated by the seven \emph{8-dimensional Pauli matrices}:
$$\begin{aligned}
  &\sigma_1=\sigma_x\otimes\sigma_x\otimes\sigma_x\\
  &\sigma_2=\sigma_x\otimes\sigma_x\otimes\epsilon\\
  &\sigma_3=\sigma_x\otimes\sigma_x\otimes\sigma_z\\
  &\sigma_4=\sigma_x\otimes\epsilon\otimes I\\
  &\sigma_5=\sigma_x\otimes\sigma_z\otimes I\\
  &\sigma_6=\epsilon\otimes I\otimes I\\
  &\sigma_7=\sigma_z\otimes I\otimes I.
\end{aligned}
$$
Using the identities \eqref{iden}, especially the one saying that $\epsilon^2=-I$, one easily finds that the seven Pauli matrices $\sigma_i$, $i=1,2,\dots,7$, satisfy the \emph{Clifford algebra identity}
$$\sigma_i\sigma_j+\sigma_j\sigma_i\,\,=\,\,2g_{ij}\,\,(I\otimes I \otimes I)\,,\quad\quad i,j=1,2,\dots,7,$$
with the coefficients  $g_{ij}$ forming a diagonal  $7\times7$ matrix
$$\Big(\,\,g_{ij}\,\,\Big)\,\,=\,\,\mathrm{diag}\Big(1,-1,1,-1,1,-1,1\Big),$$
of signature $(4,3)$. Thus, the 8-dimensional Pauli matrices $\sigma_i$, $i=1,\dots,7$, generate the Clifford algebra ${\mathcal C}\ell(4,3)$, and in turn, \emph{by the general theory}, as described in Section \ref{spintraut}, they define the \emph{spin representation} $\rho$ of $\soa(4,3)$ in an 8-dimensional real vector space $S=\bbR^8$ of Pauli(-\emph{Majorana}) spinors.

\subsubsection{The spinorial representation of $\soa(4,3)$}
To be more explicit, let $(i,j)$ be such that $1\leq i<j\leq 7$, and let $I$  be a function 
\be I(i,j)=1+i+\tfrac12(j-3)j\label{Iij}\ee
on such pairs. Note that the function $I$ is a bijection between the 21 pairs $(i,j)$ and the set of 21 natural numbers $I=1,2,\dots, 21$. Consider the twenty one $8\times 8$ real matrices $\sigma_i\sigma_j$ with $1\leq i<j\leq 7$, and a basis $\{A_I\}_{I=1}^{21}$ in the Lie algebra $\soa(4,3)$. Then the spin representation $\rho$ of $\soa(4,3)$ is given by
$$\rho(A_{I(i,j)})=\tfrac12\sigma_i\sigma_j\quad\mathrm{with}\quad 1\leq i<j\leq7.$$
Explicitly, we have:
\be
\begin{array}{lll}
  \rho(A_1)=\tfrac12 I\otimes I \otimes\sigma_z,&\quad  \rho(A_8)=\tfrac12 I\otimes \epsilon\otimes\epsilon,&\quad  \rho(A_{15})=\tfrac12 \sigma_z\otimes \sigma_z \otimes I,\\
  \rho(A_2)=\tfrac12 I\otimes I \otimes\epsilon, &\quad\rho(A_9)=\tfrac12 I\otimes \epsilon \otimes\sigma_z, &\quad  \rho(A_{16})=\tfrac12 \epsilon\otimes \sigma_x \otimes\sigma_x, \\
  \rho(A_3)=\tfrac12 I\otimes I \otimes\sigma_x, &\quad\rho(A_{10})=\tfrac12 I\otimes \sigma_x \otimes I,&\quad  \rho(A_{17})=\tfrac12 \epsilon\otimes \sigma_x \otimes\epsilon,\\
  \rho(A_4)=\tfrac12 I\otimes I \otimes\sigma_x,&\quad  \rho(A_{11})=\tfrac12 \sigma_z\otimes \sigma_x \otimes\sigma_x,&\quad  \rho(A_{18})=\tfrac12 \epsilon\otimes \sigma_x \otimes\sigma_z,\\
  \rho(A_5)=\tfrac12 I\otimes \sigma_z \otimes\epsilon,&\quad  \rho(A_{12})=\tfrac12 \sigma_z\otimes \sigma_x \otimes\epsilon,&\quad \rho(A_{19})=\tfrac12 \epsilon\otimes \epsilon  \otimes I,\\
 \rho(A_6)=\tfrac12 I\otimes \sigma_z \otimes\sigma_z,&\quad   \rho(A_{13})=\tfrac12 \sigma_z\otimes \sigma_x\otimes\sigma_z,&\quad  \rho(A_{20})=\tfrac12 \epsilon\otimes \sigma_z \otimes I,\\
 \rho(A_7)=\tfrac12 I\otimes \epsilon \otimes\sigma_x, &\quad \rho(A_{14})=\tfrac12 \sigma_z\otimes \epsilon\otimes I, &\quad  
  \rho(A_{21})=\tfrac12 \sigma_x\otimes I \otimes I.
\end{array}
\label{f41}\ee
The spin representation $\rho$ of $V_0=\bbR\oplus\soa(4,3)$ needs one generator more. Let us call its $\rho(A_{22})$. We have
$$\rho(A_{22})=\tfrac12 I\otimes I\otimes I.$$
We determine the structure constants $c^K{}_{IJ}$ of $\bbR\oplus\soa(4,3)$ in the basis $A_I$ from 
\be [\,\,\rho(A_I),\,\,\rho(A_J)\,\,]\,\,=\,\,c^K{}_{IJ}\,\,\rho(A_K).\label{strco}\ee

\subsubsection{Obtaining the vectorial representation of $\soa(4,3)$}
Now, we take the space $\bigwedge^2S$ and consider the skew symmetric representation $$\roa=\rho\wedge\rho$$ in it. We will write it in the standard basis $f_\mu$, $\mu=1,\dots, 8$ in $S=\bbR^8$. We have $\rho(A_I)f_\mu=\rho_I{}^\nu{}_\mu f_\nu$. Now, the components of the 28-dimensional representation $\roa=\rho\dz\rho_1$ are
$$\roa_{I}{}^{\mu\nu}{}_{\alpha\beta}\,\,=\,\,\,\,\rho_I{}^{\mu}{}_\alpha\delta^{\nu}{}_\beta+
\delta^{\mu}{}_\alpha\rho_I{}^{\nu}{}_\beta-\rho_I{}^{\nu}{}_\alpha\delta^{\mu}{}_\beta-\delta^{\nu}{}_\alpha\rho_I{}^{\mu}{}_\beta\,\,,$$
and we have
$$\big(\roa(A_I)w\big){}^{\mu\nu}=\roa_{I}{}^{\mu\nu}{}_{\alpha\beta}w^{\alpha\beta}, \quad \forall w^{\alpha\beta}=w^{[\alpha\beta]}.$$
The Casimir operator for this representation is
$${\mathcal C}\,\,=\,\,10\,\,K^{IJ}\,\,\roa(A_I)\,\roa(A_J),$$
where $K^{IJ}$ is the inverse of the Killing form matrix $K_{IJ}=c^L{}_{IM}c^M{}_{JL}$ in the basis $A_I$. Since for the Killing form to be nondegenerate we must restrict to the semisimple part of $\bbR\oplus\soa(4,3)$, here the indices $I,J,K,L,M=1,2\dots,21$, and as always are summed over the repeated indices. One can check that in this basis of $\soa(4,3)$ the Killing form matrix is diagonal, and reads
$$
\big(\,\, K_{IJ}\,\,\big)=
  10\,\mathrm{diag}\Big(\,1,-1,1,1,-1,1,-1,1,-1,1,1,-1,1,-1,1,-1,1,-1,1,-1,1\,\Big).
$$
The Casimir $\mathcal C$ defines the decomposition of the 28-dimensional reducible representation $\roa=\rho_1\dz\rho_1$ onto $$\textstyle \bigwedge^2S=\bigwedge_{21}\oplus\bigwedge_7,$$ where
the 7-dimensional irreducible representation \emph{space} $\bigwedge_7$ \emph{is the eigenspace of the Casimir operator consisting of eigen-bispinors with eigenvalue equal to 6},
$$\textstyle {\mathcal C}\,\,\big(\bigwedge_7\big)\,=6\,\bigwedge_7.$$
Explicitly, in the same basis $A_I$, $I=1,2,\dots,21$, as before, this 7-dimensional representation $(\tau,R=\bigwedge_7)$ of the $\soa(4,3)$ Lie algebra is given by:
\be\begin{aligned}
  &\tau(A_1)=E_{66}-E_{22},\\
  &\tau(A_2)=\tfrac12(E_{23}-E_{32}+E_{25}-E_{52}+E_{36}-E_{63}+E_{56}-E_{65}),\\
  &\tau(A_3)=\tfrac12(E_{23}+E_{32}+E_{25}+E_{52}+E_{36}+E_{63}+E_{56}+E_{65}),\\
  &\tau(A_4)=\tfrac12(E_{23}+E_{32}-E_{25}-E_{52}-E_{36}-E_{63}+E_{56}+E_{65}),\\
  &\tau(A_5)=\tfrac12(E_{23}-E_{32}-E_{25}+E_{52}-E_{36}+E_{63}+E_{56}-E_{65}),\\
  &\tau(A_6)=E_{33}-E_{55},\\
  &\tau(A_7)=\tfrac12(E_{12}-E_{21}-E_{16}+E_{61}-E_{27}+E_{72}+E_{67}-E_{76}),\\
  &\tau(A_8)=\tfrac12(-E_{12}-E_{21}-E_{16}-E_{61}-E_{27}-E_{72}-E_{67}-E_{76}),\\
  &\tau(A_9)=\tfrac12(E_{13}-E_{31}+E_{15}-E_{51}-E_{37}+E_{73}-E_{57}+E_{75}),\\
  &\tau(A_{10})=\tfrac12(E_{13}+E_{31}-E_{15}-E_{51}+E_{37}+E_{73}-E_{57}-E_{75}),\\
  &\tau(A_{11})=\tfrac12(-E_{12}-E_{21}+E_{16}+E_{61}+E_{27}+E_{72}-E_{67}-E_{76}),\\
  &\tau(A_{12})=\tfrac12(E_{12}-E_{21}+E_{16}-E_{61}+E_{27}-E_{72}+E_{67}-E_{76}),\\
  &\tau(A_{13})=\tfrac12(-E_{13}-E_{31}-E_{15}-E_{51}+E_{37}+E_{73}+E_{57}+E_{75}),\\
  &\tau(A_{14})=\tfrac12(-E_{13}+E_{31}+E_{15}-E_{51}-E_{37}+E_{73}+E_{57}-E_{75}),\\
  &\tau(A_{15})=E_{11}-E_{77},\\
  &\tau(A_{16})=\tfrac12(2E_{24}-E_{42}+E_{46}-2E_{64}),\\
  &\tau(A_{17})=\tfrac12(2E_{24}+E_{42}+E_{46}+2E_{64}),\\
  &\tau(A_{18})=\tfrac12(2E_{34}-E_{43}-E_{45}+2E_{54}),\\
  &\tau(A_{19})=\tfrac12(-2E_{34}-E_{43}+E_{45}+2E_{54}),\\
  &\tau(A_{20})=\tfrac12(-2E_{14}+E_{41}+E_{47}-2E_{74}),\\
  &\tau(A_{21})=\tfrac12(2E_{14}+E_{41}-E_{47}-2E_{74}),\\
  &\tau(A_{22})=E_{11}+E_{22}+E_{33}+E_{44}+E_{55}+E_{66}+E_{77},
  \end{aligned}\label{f411}\ee
where $E_{ij}$, $i,j=1,2,\dots,7$, denote $7\times 7$ matrices with zeroes everywhere except the value 1 in the entry $(i,j)$ seating at the crossing of the $i$th row and the $j$th column.

One can check that $$[\,\,\tau(A_I),\,\,\tau(A_J)\,\,]\,\,=\,\,c^K{}_{IJ}\,\,\tau(A_K),$$ with the same structure constants as in \eqref{strco}.

\subsubsection{A contactification with $\mathfrak{f}_I$ symmetry} 
So now we are in the situation of having two representations $(\rho,S)$ and $(\tau,R=\bigwedge_7)$ od $\mathfrak{n}_{00}=\bbR\oplus\soa(4,3)$ and we can try to solve the equation \eqref{maga} for the map $\omega\in\mathrm{Hom}(\bigwedge^2S,R)$. Of course, if we started with some arbitrary $\rho$ and $\tau$ this equation would not have solutions other than 0, but here we expect to have solution, since we know it from the Cartan's PhD thesis \cite{CartanPhd}, and the announcement in Helgason's paper \cite{He}. And indeed there is a solution for a nonzero $\omega$, which when written in the basis $\{f_\mu\}$ in $S$ and $\{e_i\}$ in $R$ is such that it gives the \emph{seven} 2-forms $\omega^i=\tfrac12\omega^i{}_{\mu\nu}\der x^\mu\dz\der x^\nu$, $i=1,\dots,7$, in $N=\bbR^8$ given by: 
\be\begin{aligned}
    \omega^1=&\der x^1\dz\der x^2-\der x^7\dz\der x^8,\\
    \omega^2=&\der x^2\dz\der x^4-\der x^6\dz\der x^8,\\
    \omega^3=&\der x^1\dz\der x^4-\der x^5\dz\der x^8,\\
    \omega^4=&\tfrac12\,\big(\,\der x^1\dz\der x^6-\der x^2\dz\der x^5-\der x^3\dz\der x^8+\der x^4\dz\der x^7\,\big),\\
    \omega^5=&\der x^2\dz\der x^3-\der x^6\dz\der x^7,\\
    \omega^6=&\der x^1\dz\der x^3-\der x^5\dz\der x^7,\\
    \omega^7=&\der x^3\dz\der x^4-\der x^5\dz\der x^6.
    \end{aligned}\label{2forms}\ee
These, via the contactification and the theory summarized in Corollary \ref{cruco} lead to the following theorem.
\begin{theorem}\label{distf1}
   Let $M=\bbR^{15}$ with coordinates $(u^1,\dots,u^7,x^1,\dots ,x^8)$, and consider seven 1-forms $\lambda^i,\dots,\lambda^7$ on $M$ given by
   $$\begin{aligned}
\lambda^1=&\der u^1+ x^1\der x^2- x^7\der x^8,\\
    \lambda^2=&\der u^2+ x^2\der x^4- x^6\der x^8,\\
    \lambda^3=&\der u^3+ x^1\der x^4- x^5\der x^8,\\
    \lambda^4=&\der u^4+\tfrac12\,\big(\, x^1\der x^6- x^2\der x^5-x^3\der x^8+ x^4\der x^7\,\big),\\
    \lambda^5=&\der u^5+x^2\der x^3- x^6\der x^7,\\
    \lambda^6=&\der u^6+x^1\der x^3- x^5\der x^7,\\
    \lambda^7=&\der u^7+x^3\der x^4- x^5\der x^6.
   \end{aligned}$$
   The rank 8 distribution ${\mathcal D}$ on $M$ defined as ${\mathcal D}=\{\mathrm{T}\bbR^{15}\ni X\,\,|\,\,X\hook\lambda^1=\dots=X\hook\lambda^7=0\}$ has its Lie algebra of infinitesimal authomorphism $\mathfrak{aut}(\mathcal D)$ isomorphic to the Tanaka prolongation of $\mathfrak{n}_{\minu}=R\oplus S$, where $(\rho,S=\bbR^8)$ is the spinorial representation \eqref{f41} of $\mathfrak{n}_{00}=\bbR\oplus \soa\big(4,3)$, and $(\tau,R=\bbR^7)$ is the vectorial representation \eqref{f411} of $\mathfrak{n}_{00}$.

  The symmetry algebra $\mathfrak{aut}({\mathcal D})$ is isomorphic to the simple Lie algebra $\mathfrak{f}_I$,
  $$\mathfrak{aut}({\mathcal D})=\mathfrak{f}_I,$$
  having the following natural gradation 
  $$\mathfrak{aut}({\mathcal D})=\mathfrak{n}_{-2}\oplus\mathfrak{n}_{-1}\oplus\mathfrak{n}_0\oplus\mathfrak{n}_1\oplus\mathfrak{n}_2,$$
  with $\mathfrak{n}_{-2}=R$, $\mathfrak{n}_{-1}=S$,
  $$
    \mathfrak{n}_0=\mathfrak{n}_{00}=\bbR\oplus\soa(4,3),$$
  $\mathfrak{n}_{1}=S^*$, $\mathfrak{n}_{2}=R^*$,
 which is inherited from the distribution structure $(M,{\mathcal D})$. The duality signs $*$ at $R^*$ and $S^*$ above are with respect to the Killing form in $\mathfrak{f}_I$.

 The contactification $(M,{\mathcal D})$ is locally a flat model for the parabolic geometry of type $(F_I,P_I)$ related to the following \emph{crossed} Satake diagram  \tikzset{/Dynkin diagram/fold style/.style={stealth-stealth,thin,
shorten <=1mm,shorten >=1mm}}\begin{dynkinDiagram}[edge length=.5cm]{F}{ooot}
 \end{dynkinDiagram}
  of $\mathfrak{f}_I$.
  \end{theorem}
\begin{remark}
Please note, that this is an example of an application of the magical equation \eqref{maga} in which the starting algebra $\mathfrak{n}_{00}$ was big enough, so that its Tanaka prolongation $\mathfrak{n}_0$ counterpart is precisely equal to $\mathfrak{n}_{00}$. This was actually expected from the construction based on the crossed Satake diagram  \tikzset{/Dynkin diagram/fold style/.style={stealth-stealth,thin,
shorten <=1mm,shorten >=1mm}}\begin{dynkinDiagram}[edge length=.5cm]{F}{ooot}
   \end{dynkinDiagram}, which shows that $\mathfrak{n}_0$ of this parabolic geometry is precisely our $\mathfrak{n}_{00}=\bbR\oplus\soa(4,3)$.  
  \end{remark}

\begin{remark}
  One sees that the distribution $\mathcal D$ in $\bbR^{15}$ with $\mathfrak{f}_I$ symmetry presented in Theorem \ref{distf1} looks different that the distribution from our Example \ref{exa4}. It follows, however that both these distributions are locally equivalent, and both have the same simple exceptional Lie algebra  $\mathfrak{f}_I$ as an algebra of their authomorphisms. 
\end{remark}

\subsubsection{Contactification for $\mathfrak{f}_I$: more algebra about $\soa(4,3)$}\label{phiform}
  In our construction of the $\mathfrak{f}_I$ symmetric distribution $\mathcal D$ in Theorem \ref{distf1} the crucial role was played by the 7-dimensional span of 2-forms $\omega^i$, $i=1,2,\dots,7$. If we were given these seven 2-forms, we would produce the $\mathfrak{f}_I$ symmetric distribution $\mathcal D$ by the procedure of contactification.

  It turns out that in $S=\bbR^8$ there is a particular 4-form
  $$\Phi=\tfrac{1}{4!}\Phi_{\mu\nu\rho\sigma}\der x^\mu\dz\der x^\nu\dz\der x^\rho\dz \der x^\sigma$$
  that is $\bbR\oplus \soa(4,3)$ invariant
 $$\rho_I{}^\alpha{}_\mu \Phi_{\alpha\nu\rho\sigma}+\rho_I{}^\alpha{}_\nu \Phi_{\mu\alpha\rho\sigma}+\rho_I{}^\alpha{}_\rho \Phi_{\mu\nu\alpha\sigma}+\rho_I{}^\alpha{}_\sigma \Phi_{\mu\nu\rho\alpha}=S_I\Phi_{\mu\nu\rho\sigma}.$$
It may be represented by:
$$\Phi=h_{ij}\omega^i\dz\omega^j,$$
where $\omega^i$ are given by \eqref{2forms} and
$$\big(\,\,h_{ij}\,\,\big)\,\,=\,\,\tfrac12\,\,
\bma
0&0&0&0&0&0&1\\0&0&0&0&0&-1&0\\0&0&0&0&1&0&0\\0&0&0&1&0&0&0\\0&0&1&0&0&0&0\\0&-1&0&0&0&0&0\\1&0&0&0&0&0&0 \ema,$$
or in words\footnote{Note that since $(h_{ij})$ is a symmetric matrix of signature $(4,3)$, this fact alone shows that the span of seven 2-forms $\omega^i$ is a $7$-dimensional representation space of $\sog(4,3)$. Actually, this fact easily leads to the construction of the double cover $\bbZ_2\to{\bf Spin}(4,3)\to\sog(4,3)$. }: $h_{ij}$, $i,j,=1,2,\dots,7$, \emph{are all zero except} $h_{17}=h_{71}=-h_{26}=-h_{62}=h_{35}=h_{53}=h_{44}=1$.

The form $\Phi$ in full beauty reads:
\be\begin{aligned}
  \tfrac23\Phi\,\,=\,\,2\,\,\Big(\,\,&\der x^1\dz\der x^2\dz\der x^3\dz\der x^4+\der x^5\dz\der x^6\dz\der x^7\dz\der x^8\,\,\Big)-\\
  &\der x^1\dz\der x^2\dz\der x^5\dz\der x^6+\der x^1\dz\der x^3\dz\der x^6\dz\der x^8-\\&\der x^1\dz\der x^4\dz\der x^6\dz\der x^7-\der x^2\dz\der x^3\dz\der x^5\dz\der x^8+\\&\der x^2\dz\der x^4\dz\der x^5\dz\der x^7-\der x^3\dz\der x^4\dz\der x^7\dz\der x^8.
\end{aligned}
\label{4form}\ee
\begin{remark}\label{4formfi}
It is remarkable that this 4-form alone encaptures all the features of the $\mathfrak{f}_I$ symmetric contactification we discussed in the entire Section \ref{41}. By this we mean following:
\begin{enumerate}
\item Consider $N=\bbR^8$ with coordinates $(x^\mu)$, $\mu=1,2,\dots,8$, and the 4-form $$\Phi=\tfrac{1}{4!}\Phi_{\mu\nu\rho\sigma}\der x^\mu\dz\der x^\nu\dz\der x^\rho\dz \der x^\sigma$$ given by \eqref{4form}.
\item Consider an equation
  $$A^\alpha{}_\mu \Phi_{\alpha\nu\rho\sigma}+A^\alpha{}_\nu \Phi_{\mu\alpha\rho\sigma}+A^\alpha{}_\rho \Phi_{\mu\nu\alpha\sigma}+A^\alpha{}_\sigma \Phi_{\mu\nu\rho\alpha}=S\Phi_{\mu\nu\rho\sigma}$$
  for the  real $7\times 7$ matrix $A=(A^\mu{}_\nu)$.
\item For simplicity solve it in two steps: \begin{itemize}
\item First with $S=0$. You obtain 21-dimensional solution space, which will be the \emph{spin representation} $\rho$ of $\soa(4,3)$. It is given $\rho(A)=A$.
  \item Then prove that the only solution with $S\neq 0$ corresponds to $S=4$, and that, modulo the addition of linear combinations of solutions with $S=0$, it is given by $A=\mathrm{Id}_{8\times 8}$. Extend your possible $A$s with $S=0$ to $A$s including $A=\mathrm{Id}_{8\times 8}$. 
\end{itemize}
\item In this way you will show that the stabilizer in $\gla(8,\bbR)$ of the 4-form $\Phi$ is the Lie algebra $\bbR\oplus\soa(4,3)$ in the \emph{spin representation} $\rho$ of Pauli spinors; $\rho(A)=A$. 
\item Then search for a 7-dimensional space of 
  2-forms, spanned say by the 2-forms  $\omega^i=\tfrac12\omega^i{}_{\mu\nu}\der x^\mu\dz\der x^\nu$ satisfying
  $$A^\alpha{}_\mu\,\, \omega^i{}_{\alpha\nu}\,\,+\,\,A^\alpha{}_\nu\,\, \omega^i{}_{\mu\alpha}\,\,=\,\,s^i{}_j \,\,\omega^j{}_{\mu\nu}\,\,$$
  for all $A$s from the spin representation $\rho(A)=A$ of $\bbR\oplus\soa(4,3)$. Here $s^i{}_j$ are auxiliary constants
  \footnote{Note however, that although you look for $\omega^i{}_{\mu\nu}$ with \emph{some} constants $s^i{}_j$, these constants have geometric meaning: comparing with our magical equation \eqref{magb} we see that the $7\times 7$ matrices $(s^i{}_j)$ constitute matrices of the defining representation $\tau$ of $\bbR\oplus\soa(4,3)$.}.
\item This space is uniquely defined by these equations, and after solving them you will get 7 linearly independent 2-forms $(\omega^1,\dots,\omega^7)$ in $N=\bbR^8$.
  \item Contactifying the resulting structure $\big(N,\Span(\omega^1,\dots,\omega^7)\big)$, as we did e.g in Theorem \ref{distf1}, you will get $\mathfrak{f}_I$ symmetric distribution ${\mathcal D}$ in $\bbR^7\to \big(\,M=\bbR^{15}\,\big)\to \big(\,N=\bbR^8\,\big)$.
\end{enumerate}
\end{remark}
\subsection{Realization of $\mathfrak{f}_{II}$}\label{42}
It seems that Cartan was only interested in the explicit realization of $\mathfrak{f}_I$. The realization of $\mathfrak{f}_{II}$ can be obtained in the same spirit as we have described in Section \ref{41}. Here without much of explanations since they parallel Section \ref{41}, we only display the main steps leading to this realization.   

 We start with the representation of the Clifford algebra ${\mathcal C}\ell(0,7)$ generated by the seven $\rho$-matrices from \eqref{cl07}. They satisfy
 $$\rho_i\rho_j+\rho_j\rho_i=-2\delta_{ij}I_{8\times 8},\quad i,j=1,\dots,7.$$
 They induce the 8-dimensional representation
 $$\rho:\bbR\oplus\soa(0,7)\to \mathrm{End}(S)$$
   of $n_{00}=\bbR\oplus\soa(0,7)$ in the space $S=\bbR^8$ of real Pauli spinors, generated by the 22 real $8\times 8$ matrices:
$$\begin{aligned}
     &\rho(A_{I(i,j)})=\tfrac12\rho_i\rho_j, \quad 1\leq i<j\leq 7,\\
     &\rho(A_{22})=\tfrac12 (I\otimes I\otimes I),
   \end{aligned}
   $$
   with the index $I=I(i,j)$ given by \eqref{Iij}, and with $I,\sigma_x,\epsilon,\sigma_z$ given by \eqref{pauu}-\eqref{pauu1}.
   Explicitly, in terms of matrices $I,\sigma_x,\epsilon,\sigma_z$ the generators of this spinorial representation of $\soa(0,7)$ are:
\be
\begin{array}{lll}
  \rho(A_1)=-\tfrac12 I\otimes \epsilon \otimes\sigma_z,&\quad  \rho(A_8)=\tfrac12 \epsilon\otimes \sigma_z\otimes\sigma_z,&\quad  \rho(A_{15})=-\tfrac12 I\otimes \epsilon \otimes I,\\
  \rho(A_2)=\tfrac12 I\otimes \epsilon \otimes\sigma_x, &\quad\rho(A_9)=-\tfrac12 \epsilon\otimes \sigma_z \otimes\sigma_x, &\quad  \rho(A_{16})=-\tfrac12 \sigma_x\otimes I \otimes\epsilon, \\
  \rho(A_3)=-\tfrac12 I\otimes I \otimes\epsilon, &\quad\rho(A_{10})=-\tfrac12 I\otimes \sigma_z \otimes \epsilon,&\quad  \rho(A_{17})=-\tfrac12 \sigma_x\otimes\epsilon \otimes\sigma_x,\\
  \rho(A_4)=-\tfrac12 \epsilon\otimes \epsilon\otimes\epsilon,&\quad  \rho(A_{11})=\tfrac12 \epsilon\otimes \sigma_z \otimes I,&\quad  \rho(A_{18})=-\tfrac12 \sigma_x\otimes \epsilon \otimes\sigma_z,\\
  \rho(A_5)=\tfrac12 \epsilon\otimes I \otimes\sigma_x,&\quad  \rho(A_{12})=-\tfrac12 \epsilon\otimes \sigma_x \otimes\sigma_z,&\quad \rho(A_{19})=\tfrac12 \sigma_z\otimes \epsilon  \otimes I,\\
 \rho(A_6)=\tfrac12 \epsilon\otimes I\otimes\sigma_z,&\quad   \rho(A_{13})=\tfrac12 \epsilon\otimes \sigma_x\otimes\sigma_x,&\quad  \rho(A_{20})=\tfrac12 \sigma_z\otimes \sigma_x \otimes \epsilon,\\
 \rho(A_7)=\tfrac12 \epsilon\otimes \sigma_x \otimes I, &\quad \rho(A_{14})=\tfrac12 I \otimes \sigma_x\otimes \epsilon, &\quad  
  \rho(A_{21})=\tfrac12 \sigma_z\otimes \sigma_z \otimes \epsilon.
\end{array}
\label{f42}\ee
We also write down the corresponding generators of the vectorial representation $\tau$, which is the 7-dimensional irreducible component $\bigwedge_7$ of the representation $\rho\dz\rho$,
which decomposes as $\bigwedge^2S=\bigwedge_{21}\oplus\bigwedge_7$. These generators read:
\be
\begin{array}{lll}
  \tau(A_1)=E_{31}-E_{13},&\quad  \tau(A_8)=E_{37}-E_{73},&\quad  \tau(A_{15})=E_{75}-E_{57},\\
  \tau(A_2)=E_{12}-E_{21}, &\quad\tau(A_9)=E_{72}-E_{27}, &\quad  \tau(A_{16})=E_{14}-E_{41}, \\
  \tau(A_3)=E_{32}-E_{23}, &\quad\tau(A_{10})=E_{76}-E_{67},&\quad  \tau(A_{17})=E_{34}-E_{43},\\
  \tau(A_4)=E_{61}-E_{16},&\quad  \tau(A_{11})=E_{51}-E_{15},&\quad  \tau(A_{18})=E_{42}-E_{24},\\
  \tau(A_5)=E_{36}-E_{63},&\quad  \tau(A_{12})=E_{53}-E_{35},&\quad \tau(A_{19})=E_{46}-E_{64},\\
 \tau(A_6)=E_{62}-E_{26},&\quad   \tau(A_{13})=E_{25}-E_{52},&\quad  \tau(A_{20})=E_{47}-E_{74},\\
 \tau(A_7)=E_{17}-E_{71}, &\quad \tau(A_{14})= E_{65}-E_{56},&\quad  
  \tau(A_{21})=E_{54}-E_{45},
\end{array}
\label{f421}\ee
where $E_{ij}$ are $7\times 7$ matrices with all zero entries, except at the $i$th-$j$th entry, where 1 resides.

We are again in a position ready for application of our Lemma \ref{l21}. Given the representations $(\rho,S=\bbR^8)$ and $(\tau,R=\bigwedge_7)$ of $\soa(0,7)$ we solve the magical equation \eqref{maga} for $\omega=\tfrac12\omega^i{}_{\mu\nu}e_i\otimes f^\mu\dz f^\nu$.  In this way we obtain the seven 2-forms $\omega^i=-\tfrac12\omega^i{}_{\mu\nu}\der x^\mu\dz\der x^\nu$ on $N=\bbR^8$, with coordinates $(x^\mu)_{\mu=1}^8$, which read as follows:
\be\begin{aligned}
\omega^1=&-\der x^1\dz\der x^2-\der x^3\dz\der x^4+\der x^5\dz\der x^6+\der x^7\dz\der x^8,\\
\omega^2=&\der x^1\dz\der x^3-\der x^2\dz\der x^4-\der x^5\dz\der x^7+\der x^6\dz\der x^8,\\
\omega^3=&-\der x^1\dz\der x^4-\der x^2\dz\der x^3+\der x^5\dz\der x^8+\der x^6\dz\der x^7,\\
\omega^4=&\der x^1\dz\der x^5+\der x^2\dz\der x^6+\der x^3\dz\der x^7+\der x^4\dz\der x^8,\\
\omega^5=&-\der x^1\dz\der x^6+\der x^2\dz\der x^5+\der x^3\dz\der x^8-\der x^4\dz\der x^7,\\
\omega^6=&\der x^1\dz\der x^7+\der x^2\dz\der x^8-\der x^3\dz\der x^5-\der x^4\dz\der x^6,\\
 \omega^7=&\der x^1\dz\der x^8-\der x^2\dz\der x^7+\der x^3\dz\der x^6-\der x^4\dz\der x^5.
   \end{aligned}\label{2forms2}\ee
These, via the contactification lead to the following theorem.
\begin{theorem}\label{distf2}
   Let $M=\bbR^{15}$ with coordinates $(u^1,\dots,u^7,x^1,\dots ,x^8)$, and consider seven 1-forms $\lambda^i,\dots,\lambda^7$ on $M$ given by
   $$\begin{aligned}
\lambda^1=&\der u^1- x^1\der x^2- x^3\der x^4+ x^5\der x^6+ x^7\der x^8,\\
\lambda^2=&\der u^2+ x^1\der x^3- x^2\der x^4-x^5\der x^7+x^6\der x^8,\\
\lambda^3=&\der u^3- x^1\der x^4- x^2\der x^3+ x^5\der x^8+ x^6\der x^7,\\
\lambda^4=&\der u^4+ x^1\der x^5+ x^2\der x^6+ x^3\der x^7+ x^4\der x^8,\\
\lambda^5=&\der u^5- x^1\der x^6+x^2\der x^5+x^3\der x^8-x^4\der x^7,\\
\lambda^6=&\der u^6+ x^1\der x^7+ x^2\der x^8- x^3\der x^5-x^4\der x^6,\\
 \lambda^7=&\der u^7+ x^1\der x^8- x^2\der x^7+ x^3\der x^6 - x^4\der x^5.
   \end{aligned}$$
   The rank 8 distribution ${\mathcal D}$ on $M$ defined as ${\mathcal D}=\{\mathrm{T}\bbR^{15}\ni X\,\,|\,\,X\hook\lambda^1=\dots=X\hook\lambda^7=0\}$ has its Lie algebra of infinitesimal authomorphisms $\mathfrak{aut}(\mathcal D)$ isomorphic to the Tanaka prolongation of $\mathfrak{n}_{\minu}=R\oplus S$, where $(\rho,S=\bbR^8)$ is the spinorial representation \eqref{f42} of $\mathfrak{n}_{00}=\bbR\oplus \soa\big(0,7)$, and $(\tau,R=\bbR^7)$ is the vectorial representation \eqref{f421} of $\mathfrak{n}_{00}$.

  The symmetry algebra $\mathfrak{aut}({\mathcal D})$ is isomorphic to the simple Lie algebra $\mathfrak{f}_{II}$,
  $$\mathfrak{aut}({\mathcal D})=\mathfrak{f}_{II},$$
  having the following natural gradation 
  $$\mathfrak{aut}({\mathcal D})=\mathfrak{n}_{-2}\oplus\mathfrak{n}_{-1}\oplus\mathfrak{n}_0\oplus\mathfrak{n}_1\oplus\mathfrak{n}_2,$$
  with $\mathfrak{n}_{-2}=R$, $\mathfrak{n}_{-1}=S$,
  $$
    \mathfrak{n}_0=\mathfrak{n}_{00}=\bbR\oplus\soa(0,7),$$
  $\mathfrak{n}_{1}=S^*$, $\mathfrak{n}_{2}=R^*$,
 which is inherited from the distribution structure $(M,{\mathcal D})$. The duality signs $*$ at $R^*$ and $S^*$ above are with respect to the Killing form in $\mathfrak{f}_{II}$.

 The contactification $(M,{\mathcal D})$ is locally a flat model for the parabolic geometry of type $(F_{II},P_{II})$ related to the following \emph{crossed} Satake diagram  \tikzset{/Dynkin diagram/fold style/.style={stealth-stealth,thin,
shorten <=1mm,shorten >=1mm}}\begin{dynkinDiagram}[edge length=.5cm]{F}{***t}
 \end{dynkinDiagram}
  of $\mathfrak{f}_{II}$.
  \end{theorem}
\begin{remark} In this way we realized the real form $\mathfrak{f}_{II}$ of the simple exceptional complex Lie algebra $\mathfrak{f}_4$ in $M=\bbR^{15}$ as a symmetry algebra of the Pfaffian system $(\lambda^1,\dots,\lambda^7)$. 
  This realization does not appear in Cartan's theses.  
  \end{remark}
\begin{remark}
  Our present case of $\mathfrak{f}_{II}$ also admits description in terms of  a certain $\bbR\oplus\soa(0,7)$ invariant 4-form $\Phi$ in $S=\bbR^8$, analogous to the 4-form $\Phi$ introduced in Section \ref{phiform}, when we discussed $\mathfrak{f}_I$. Skipping the details we only mention that
  now $\Phi$ may be represented by:
$$\Phi=h_{ij}\omega^i\dz\omega^j,$$
where $\omega^i$ are given by \eqref{2forms2} and
$$\big(\,\,h_{ij}\,\,\big)=
\bma
-1&0&0&0&0&0&0\\0&-1&0&0&0&0&0\\0&0&-1&0&0&0&0\\0&0&0&-1&0&0&0\\0&0&0&0&-1&0&0\\0&0&0&0&0&-1&0\\0&0&0&0&0&0&-1 \ema.$$

Explicitly, the form $\Phi$ reads:
\be\begin{aligned}
  -\tfrac16\Phi\,\,=\,\,&\der x^1\dz\der x^2\dz\der x^3\dz\der x^4-\der x^1\dz\der x^2\dz\der x^5\dz\der x^6-\\
  &\der x^1\dz\der x^2\dz\der x^7\dz\der x^8-\der x^1\dz\der x^3\dz\der x^5\dz\der x^7+\\&\der x^1\dz\der x^3\dz\der x^6\dz\der x^8-\der x^1\dz\der x^4\dz\der x^5\dz\der x^8-\\&\der x^1\dz\der x^4\dz\der x^6\dz\der x^7-\der x^2\dz\der x^3\dz\der x^5\dz\der x^8-\\
  &\der x^2\dz\der x^3\dz\der x^6\dz\der x^7+\der x^2\dz\der x^4\dz\der x^5\dz\der x^7-\\&\der x^2\dz\der x^4\dz\der x^6\dz\der x^8-\der x^3\dz\der x^4\dz\der x^5\dz\der x^6-\\&\der x^3\dz\der x^4\dz\der x^7\dz\der x^8+\der x^5\dz\der x^6\dz\der x^7\dz\der x^8.
\end{aligned}
\label{4formfii}\ee
This 4-form alone encaptures all the features of the $\mathfrak{f}_{II}$ symmetric contactification we discussed in the entire Section \ref{42}. In particular analogous statements as in Remark \ref{4formfi}, with now $\soa(4,3)$ replaced by $\soa(0,7)$, apply to the present 4-form $\Phi$.
  \end{remark}

\section{Spinorial representations in dimension 8}
Dimension \emph{eight} is quite exceptional, as for example, 8 is the highest possible dimension for the existence of Euclidean Hurwitz algebras, gifting us with the algebra of \emph{octonions}. From the perspective of our paper, which meanders through the realm of simple Lie algebras, eight is \emph{very} special: among all the complex simple Lie algebras, the Dynkin diagram of $\mathfrak{d}_4=\soa(8,\bbC)$ which is \emph{defined} in dimension \emph{eight}, is the most symmetric:\\
\centerline{\begin{dynkinDiagram}[edge length=.4cm]{D}{oooo}
  \end{dynkinDiagram}.}
\noindent
Visibly it has a threefold symmetry $S_3$.

The Lie algebra $\soa(8,\bbC)$ has six real forms. These are: $\soa(8,0)$, $\soa(7,1)$, $\soa(6,2)$, $\soa^*(8)$, $\soa(5,3)$ and $\soa(4,4)$, with the following respective Satake diagrams:\\
\centerline{\begin{dynkinDiagram}[edge length=.4cm]{D}{****}
\end{dynkinDiagram},\hspace{0.5cm}\begin{dynkinDiagram}[edge length=.4cm]{D}{o***}
\end{dynkinDiagram},\hspace{0.5cm}\begin{dynkinDiagram}[edge length=.4cm]{D}{oo**}
\end{dynkinDiagram},\hspace{0.5cm}\begin{dynkinDiagram}[edge length=.4cm]{D}{ooo*}
\end{dynkinDiagram},\hspace{0.5cm}\tikzset{/Dynkin diagram/fold style/.style={stealth-stealth,thin,
shorten <=1mm,shorten >=1mm}}\begin{dynkinDiagram}[edge length=.4cm]{D}{oooo}
\dynkinFold{3}{4}
  \end{dynkinDiagram},\hspace{0.5cm}\begin{dynkinDiagram}[edge length=.4cm]{D}{oooo}\end{dynkinDiagram}.}
We see that among these Satake diagrams the only ones that share the $S_3$ symmetry of the Dynkin diagram of the complex algebra $\mathfrak{d}_4$ are those of the \emph{compact real form} $\soa(8,0)$ and of the \emph{split real form} $\soa(4,4)$.

This $S_3$ symmetry of these two diagrams, indicates that the lowest dimensional real representations of $\soa(4,4)$ and $\soa(8,0)$, may have additional features when compared with spinorial representations of other $\soa(p,q)$s. In particular, for \emph{both} $\soa(4,4)$ and $\soa(8,0)$ we have:
\begin{itemize}
\item Their Dirac representation $(\rho,S)$ in the real vector space $S=\bbR^{16}$ is reducible over $\bbR$ and its split into two real \emph{Weyl} representations $(\rho_+,S_+)$ and $(\rho_-,S_-)$ in the respective vector spaces of \emph{Weyl spinors} $S_+=\bbR^8$ and $S_-=\bbR^8$, which have the same real dimension \emph{eight},
  $$\rho=\rho_+\oplus\rho_-\quad\mathrm{in}\quad S=S_+\oplus S_-, \quad \mathrm{dim}_\bbR S_\pm=8.$$
\item The real Weyl representations $(\rho_\pm,S_\pm)$ are \emph{faithful}, \emph{irreducible} and \emph{nonequivalent}.
\item The defining representations $(\tau,R)$ of $\soa(4,4)$ and $\soa(8,0)$, as the algebra of endomorphisms in the space $R=\bbR^8$ of vectors preserving the bilinear form of respective signatures  $(4,4)$ and $(8,0)$ has the same dimension \emph{eight} as the two Weyl representations $(\rho_\pm,S_\pm)$.
\item The real defining representations $(\tau,R)$ are \emph{irreducible} for both $\soa(4,4)$ and $\soa(8,0)$.
  \item All three real 8-dimensional irreducible representations $(\rho_+,S_+)$, $(\rho_-,S_-)$ and $(\tau,R)$ of, respectively both, $\soa(4,4)$ and $\soa(8,0)$ are \emph{pairwise nonequivalent}.
\end{itemize}

Thus the Lie algebras $\soa(4,4)$ and $\soa(8,0)$ have three real, irreducible and nonequivalent representations $(\rho_+,\rho_-,\tau)$ in the vector space $\bbR^8$ of the \emph{defining} dimension $p+q=8$. For all $\soa(p,q)$ Lie algebras this is the only dimension $p+q$ that such situation with the irreducible representations occurs.

Below, we provide the explicit description of the \emph{triality  representations}  $(\rho_+,\rho_-,\tau)$ separately for $\soa(4,4)$ and $\soa(8,0)$. 

\subsection{Triality representations of $\soa(4,4)$}
We recall from Section \ref{spintraut} that the Lie algebra $\soa(4,4)$ admits a representation $\rho$ in the 16-dimensional real vector space $S=\bbR^{16}$ of Dirac spinors. This is obtained by using the Dirac $\gamma$ matrices generating the representation of the Clifford algebra ${\mathcal C}\ell(4,4)$. In terms of the 2-dimensional Pauli matrices $(\sigma_x,\epsilon,\sigma_z,I)$ these look as follows:
  \be \begin{aligned}
     &\gamma_1=\sigma_x\otimes\sigma_x\otimes\sigma_x\otimes\sigma_x\\
  &\gamma_2=\sigma_x\otimes\sigma_x\otimes\sigma_x\otimes\epsilon\\
  &\gamma_3=\sigma_x\otimes\sigma_x\otimes\sigma_x\otimes\sigma_z\\
  &\gamma_4=\sigma_x\otimes\sigma_x\otimes\epsilon\otimes I\\
  &\gamma_5=\sigma_x\otimes\sigma_x\otimes\sigma_z\otimes I\\
  &\gamma_6=\sigma_x\otimes\epsilon\otimes I\otimes I\\
    &\gamma_7=\sigma_x\otimes\sigma_z\otimes I\otimes I\\
     &\gamma_8=\epsilon\otimes I\otimes I\otimes I.
    \end{aligned}\label{dirga}\ee
They satisfy the \emph{Dirac identity} 
\be \gamma_i\gamma_j+\gamma_j\gamma_i=2g_{ij} (I\otimes I\otimes I\otimes I), \quad i,j=1,\dots,8,\label{clifi}\ee
with
$$\big( g_{ij} \big)=\mathrm{diag}(1,-1,1,-1,1,-1,1,-1).$$

The 28 generators of $\soa(4,4)$ in the Majorana-Dirac spinor representation $\rho$ in the space of Dirac spinors $S=\bbR^{16}$ are given by
$$\rho(A_{I(i,j)})=\tfrac12\gamma_i\gamma_j, \quad 1\leq i<j\leq 8,$$
where we again have used the function $I=I(i,j)$ defined in \eqref{Iij}. Note that since now $i<j$ can run from 1 to 8, the function has a range from 1 to 28.
We add to these generators the scaling generator, $\rho(A_{29})$,
$$\rho(A_{29})=\tfrac12 I\otimes I\otimes I\otimes I.$$
This extends the Dirac representation $\rho$ of the Lie algebra $\soa(4,4)$ to the representation of the \emph{homothety} Lie algebra $\mathfrak{coa}(4,4)=\bbR\oplus\soa(4,4)$.

In terms of the 2-dimensional Pauli matrices these generators look like:
\be\label{dir44}
\begin{array}{ll}
  \rho(A_1)= \tfrac12 I \otimes I\otimes I\otimes \sigma_z,&\rho(A_{15})=\tfrac12 I \otimes \sigma_z\otimes \sigma_z\otimes I,\\
  \rho(A_2)= \tfrac12  I \otimes I\otimes I\otimes \epsilon ,&\rho(A_{16})=\tfrac12 I \otimes \epsilon\otimes \sigma_x\otimes \sigma_x,\\
  \rho(A_3)= \tfrac12  I \otimes I\otimes I\otimes \sigma_x ,&\rho(A_{17})=\tfrac12  I \otimes \epsilon\otimes \sigma_x\otimes \epsilon ,\\
  \rho(A_4)=\tfrac12   I \otimes I\otimes \sigma_z\otimes \sigma_x,&\rho(A_{18})=\tfrac12  I \otimes \epsilon\otimes \sigma_x\otimes \sigma_z ,\\
  \rho(A_5)= \tfrac12  I \otimes I\otimes \sigma_z\otimes \epsilon,&\rho(A_{19})=\tfrac12  I \otimes \epsilon\otimes \epsilon\otimes I ,\\
  \rho(A_6)= \tfrac12  I \otimes I\otimes \sigma_z\otimes \sigma_z,&\rho(A_{20})=\tfrac12 I \otimes \epsilon\otimes \sigma_z\otimes I ,\\
  \rho(A_7)= \tfrac12  I \otimes I\otimes \epsilon\otimes \sigma_x,&\rho(A_{21})=\tfrac12 I \otimes \sigma_x\otimes I\otimes I,\\
  \rho(A_8)= \tfrac12 I \otimes I\otimes \epsilon\otimes \epsilon,&\rho(A_{22})=\tfrac12 \sigma_z \otimes \sigma_x\otimes \sigma_x\otimes \sigma_x ,\\
  \rho(A_9)= \tfrac12 I \otimes I\otimes \epsilon\otimes \sigma_z,&\rho(A_{23})=\tfrac12 \sigma_z \otimes \sigma_x\otimes \sigma_x\otimes \epsilon  ,\\
  \rho(A_{10})=\tfrac12  I \otimes I\otimes \sigma_x\otimes I,&\rho(A_{24})=\tfrac12 \sigma_z \otimes \sigma_x\otimes \sigma_x\otimes \sigma_z ,\\
  \rho(A_{11})= \tfrac12  I \otimes \sigma_z\otimes \sigma_x\otimes \sigma_x ,&\rho(A_{25})=\tfrac12 \sigma_z \otimes \sigma_x\otimes \epsilon\otimes I  ,\\
  \rho(A_{12})= \tfrac12 I \otimes \sigma_z\otimes \sigma_x\otimes \epsilon,&\rho(A_{26})=\tfrac12  \sigma_z \otimes \sigma_x\otimes \sigma_z\otimes I ,\\
  \rho(A_{13})= \tfrac12  I \otimes \sigma_z\otimes \sigma_x\otimes \sigma_z,&\rho(A_{27})=\tfrac12  \sigma_z \otimes \epsilon\otimes I\otimes I,\\
   \rho(A_{14})= \tfrac12  I \otimes \sigma_z\otimes \epsilon\otimes I,&\rho(A_{28})=\tfrac12  \sigma_z \otimes \sigma_z\otimes I\otimes I.
\end{array}
\ee
Looking at the first factor in \emph{all} of these generators we observe that it is either $I$ or $\sigma_z$, i.e. it is \emph{diagonal}. This means that this 16-dimensional representation of $\bbR\oplus\soa(4,4)$ is \emph{reducible}. It \emph{splits} onto two real $8$-dimensional \emph{Weyl representations}
$$\rho=\rho_+\oplus\rho_-\quad\mathrm{in}\quad S=S_+\oplus S_-, \quad \mathrm{dim}_\bbR S_\pm=8,$$
in the spaces $S_\pm$ of (Majorana)-Weyl spinors.

On generators of $\soa(4,4)$ these two 8-dimensional representations $\rho_\pm$, are given by:
\be
\begin{array}{ll}
  \rho_\pm(A_1)= \tfrac12  I\otimes I\otimes \sigma_z,&\rho_\pm(A_{15})=\tfrac12  \sigma_z\otimes \sigma_z\otimes I,\\
  \rho_\pm(A_2)= \tfrac12   I\otimes I\otimes \epsilon ,&\rho_\pm(A_{16})=\tfrac12  \epsilon\otimes \sigma_x\otimes \sigma_x,\\
  \rho_\pm(A_3)= \tfrac12   I\otimes I\otimes \sigma_x ,&\rho_\pm(A_{17})=\tfrac12  \epsilon\otimes \sigma_x\otimes \epsilon ,\\
  \rho_\pm(A_4)=\tfrac12    I\otimes \sigma_z\otimes \sigma_x,&\rho_\pm(A_{18})=\tfrac12  \epsilon\otimes \sigma_x\otimes \sigma_z ,\\
  \rho_\pm(A_5)= \tfrac12   I\otimes \sigma_z\otimes \epsilon,&\rho_\pm(A_{19})=\tfrac12  \epsilon\otimes \epsilon\otimes I ,\\
  \rho_\pm(A_6)= \tfrac12  I\otimes \sigma_z\otimes \sigma_z,&\rho_\pm(A_{20})=\tfrac12  \epsilon\otimes \sigma_z\otimes I ,\\
  \rho_\pm(A_7)= \tfrac12   I\otimes \epsilon\otimes \sigma_x,&\rho_\pm(A_{21})=\tfrac12  \sigma_x\otimes I\otimes I,\\
  \rho_\pm(A_8)= \tfrac12 I\otimes \epsilon\otimes \epsilon,&\rho_\pm(A_{22})=\pm\tfrac12  \sigma_x\otimes \sigma_x\otimes \sigma_x ,\\
  \rho_\pm(A_9)= \tfrac12 I\otimes \epsilon\otimes \sigma_z,&\rho_\pm(A_{23})=\pm\tfrac12 \sigma_x\otimes \sigma_x\otimes \epsilon  ,\\
  \rho_\pm(A_{10})=\tfrac12  I\otimes \sigma_x\otimes I,&\rho_\pm(A_{24})=\pm\tfrac12 \sigma_x\otimes \sigma_x\otimes \sigma_z ,\\
  \rho_\pm(A_{11})= \tfrac12   \sigma_z\otimes \sigma_x\otimes \sigma_x ,&\rho_\pm(A_{25})=\pm\tfrac12  \sigma_x\otimes \epsilon\otimes I  ,\\
  \rho_\pm(A_{12})= \tfrac12 \sigma_z\otimes \sigma_x\otimes \epsilon,&\rho_\pm(A_{26})=\pm\tfrac12    \sigma_x\otimes \sigma_z\otimes I ,\\
  \rho_\pm(A_{13})= \tfrac12  \sigma_z\otimes \sigma_x\otimes \sigma_z,&\rho_\pm(A_{27})=\pm\tfrac12   \epsilon\otimes I\otimes I,\\
   \rho_\pm(A_{14})= \tfrac12  \sigma_z\otimes \epsilon\otimes I,&\rho_\pm(A_{28})=\pm\tfrac12    \sigma_z\otimes I\otimes I.
\end{array}\label{rhopm}
\ee
We extend them to $\bbR\oplus\soa(4,4)$ by adding
$$\rho_\pm(A_{29})=\tfrac12 I\otimes I\otimes I.$$
It follows that the Weyl representations $(\rho_\pm,S_\pm)$ of $\soa(4,4)$ are \emph{irreducible} and \emph{nonequivalent}.

They can be used to find yet another real 8-dimensional representation of $\soa(4,4)$. For this one considers the tensor product representation $$\rho_+\otimes\rho_-.$$ This 64-dimensional real representation of $\soa(4,4)$ is \emph{reducible}. It decomposes as:
$$\rho_+\otimes\rho_-=\alpha\oplus\tau\quad\mathrm{in}\quad S_+\otimes S_-=T_{56}\oplus R,\quad\mathrm{with}\quad \dim_\bbR(R)=8,\,\,\dim_\bbR(T_{56})=56,$$
having irreducible components $(\alpha,T_{56})$ and $(\tau,R)$ of respective dimensions 56 and 8. Explicitly, on generators of $\bbR\oplus\soa(4,4)$, the 8-dimensional representation $\tau$ reads:
\be\begin{aligned}
  &\tau(A_1)=E_{66}-E_{22},\\
  &\tau(A_2)=\tfrac12(E_{23}-E_{32}+E_{25}-E_{52}+E_{36}-E_{63}+E_{56}-E_{65}),\\
  &\tau(A_3)=\tfrac12(E_{23}+E_{32}+E_{25}+E_{52}+E_{36}+E_{63}+E_{56}+E_{65}),\\
  &\tau(A_4)=\tfrac12(E_{23}+E_{32}-E_{25}-E_{52}-E_{36}-E_{63}+E_{56}+E_{65}),\\
  &\tau(A_5)=\tfrac12(E_{23}-E_{32}-E_{25}+E_{52}-E_{36}+E_{63}+E_{56}-E_{65}),\\
  &\tau(A_6)=E_{33}-E_{55},\\
  &\tau(A_7)=\tfrac12(E_{12}-E_{21}-E_{16}+E_{61}-E_{27}+E_{72}+E_{67}-E_{76}),\\
  &\tau(A_8)=\tfrac12(-E_{12}-E_{21}-E_{16}-E_{61}-E_{27}-E_{72}-E_{67}-E_{76}),\\
  &\tau(A_9)=\tfrac12(E_{13}-E_{31}+E_{15}-E_{51}-E_{37}+E_{73}-E_{57}+E_{75}),\\
  &\tau(A_{10})=\tfrac12(E_{13}+E_{31}-E_{15}-E_{51}+E_{37}+E_{73}-E_{57}-E_{75}),\\
  &\tau(A_{11})=\tfrac12(-E_{12}-E_{21}+E_{16}+E_{61}+E_{27}+E_{72}-E_{67}-E_{76}),\\
  &\tau(A_{12})=\tfrac12(E_{12}-E_{21}+E_{16}-E_{61}+E_{27}-E_{72}+E_{67}-E_{76}),\\
  &\tau(A_{13})=\tfrac12(-E_{13}-E_{31}-E_{15}-E_{51}+E_{37}+E_{73}+E_{57}+E_{75}),\\
  &\tau(A_{14})=\tfrac12(-E_{13}+E_{31}+E_{15}-E_{51}-E_{37}+E_{73}+E_{57}-E_{75}),\\
  &\tau(A_{15})=E_{11}-E_{77},\\
  &\tau(A_{16})=\tfrac12(E_{24}-E_{42}+E_{28}-E_{82}+E_{46}-E_{64}-E_{68}+E_{86}),\\
  &\tau(A_{17})=\tfrac12(E_{24}+E_{42}+E_{28}+E_{82}+E_{46}+E_{64}+E_{68}+E_{86}),\\
  &\tau(A_{18})=\tfrac12(E_{34}-E_{43}+E_{38}-E_{83}-E_{45}+E_{54}+E_{58}-E_{85}),\\
  &\tau(A_{19})=\tfrac12(-E_{34}-E_{43}-E_{38}-E_{83}+E_{45}+E_{54}+E_{58}+E_{85}),\\
  &\tau(A_{20})=\tfrac12(-E_{14}+E_{41}-E_{18}+E_{81}+E_{47}-E_{74}-E_{78}+E_{87}),\\
&\tau(A_{21})=\tfrac12(E_{14}+E_{41}+E_{18}+E_{81}-E_{47}-E_{74}-E_{78}-E_{87}),\\
&\tau(A_{22})=\tfrac12(-E_{24}-E_{42}+E_{28}+E_{82}+E_{46}+E_{64}-E_{68}-E_{86}),\\
&\tau(A_{23})=\tfrac12(-E_{24}+E_{42}+E_{28}-E_{82}+E_{46}-E_{64}+E_{68}-E_{86}),\\
&\tau(A_{24})=\tfrac12(-E_{34}-E_{43}+E_{38}+E_{83}-E_{45}-E_{54}+E_{58}+E_{85}),\\
&\tau(A_{25})=\tfrac12(E_{34}-E_{43}-E_{38}+E_{83}+E_{45}-E_{54}+E_{58}-E_{85}),\\
&\tau(A_{26})=\tfrac12(E_{14}+E_{41}-E_{18}-E_{81}+E_{47}+E_{74}-E_{78}-E_{87}),\\
&\tau(A_{27})=\tfrac12(-E_{14}+E_{41}+E_{18}-E_{81}-E_{47}+E_{74}-E_{78}+E_{87}),\\
 &\tau(A_{28})=-E_{44}+E_{88},\\
  &\tau(A_{29})=E_{11}+E_{22}+E_{33}+E_{44}+E_{55}+E_{66}+E_{77}+E_{88},
  \end{aligned}\label{tauweyl}\ee
where $E_{ij}$, $i,j=1,2,\dots,8$, denote $8\times 8$ matrices with zeroes everywhere except the value 1 in the entry $(i,j)$ seating at the crossing of the $i$th row and the $j$th column.

 The three real, irreducible, pairwise nonequivalent representations $(\rho_+,\rho_-,\tau)$ of $\soa(4,4)$, given by the formulas \eqref{rhopm} and \eqref{tauweyl} constitute the set of the \emph{triality representations} for $\soa(4,4)$.   

 \subsection{Triality representations of $\soa(8,0)$}
 To get the real representation $(\rho,S)$ of $\soa(8,0)$ in the space $S=\bbR^{16}$ of Dirac spinors we need the real Dirac $\gamma$ matrices satisfying the \emph{Dirac identity} \eqref{clifi}, but now  with $$g_{ij}=\delta_{ij},$$ where $\delta_{ij}$ is the \emph{Kronecker delta} in 8 dimensions.

 Thus we need to modify the Dirac matrices $\gamma_i$ from \eqref{dirga} to have the proper signature of the metric. This is done in \emph{two} steps \cite{traut}. \emph{First} one puts the \emph{imaginary unit} $i$ in front of some of the Dirac matrices $\gamma_i$ generating the Clifford algebra ${\mathcal C}\ell(4,4)$, to get the proper signature of $(g_{ij})$. Although this produces few \emph{complex} generators, \emph{in step two} one uses them with the others, and modifies them in an algorithmic fashion so that they become all real and still satisfy the \emph{Dirac identity} with the proper signature of $(g_{ij})$. Explicitly, it is done as follows:

 By placing the \emph{imaginary unit} $i$ in front of $\gamma_2$, $\gamma_4$, $\gamma_6$ and $\gamma_8$ in \eqref{dirga} we obtain 8 matrices
 $$\tilde{\gamma}_{2j-1}=\gamma_{2j-1}, \quad \tilde{\gamma}_{2j}=i\gamma_{2j}, \quad j=1,2,3,4,$$
 with $\gamma_i$, $i,1,\dots,8$, in \eqref{dirga}. These constitute generators of the \emph{complex} 16-dimensional representation of the Clifford algebra ${\mathcal C}\ell(8,0)$. We will also need the representation of this Clifford algebra, which is \emph{complex conjugate} of $\tilde{\gamma}$. This is generated by
 $$\overline{\tilde{\gamma}}_{2j-1}=\gamma_{2j-1}, \quad \overline{\tilde{\gamma}}_{2j}=-\gamma_{2j}, \quad j=1,2,3,4.$$
 The Clifford algebra representations generated by the Dirac matrices $\tilde{\gamma}$ and   $\overline{\tilde{\gamma}}$ are \emph{real equivalent}, i.e. there exists a real $16\times 16$ matrix $B$ such that
 $$B\tilde{\gamma}_i=\overline{\tilde{\gamma}}_iB,\quad \forall i=1,\dots,8.$$
 It can be chosen so that
 $$B^2=\mathrm{Id},$$
 where $\mathrm{Id}=I\otimes I\otimes I\otimes I$.

 Explicitly,
 $$B=\sigma_z\otimes\epsilon\otimes\sigma_z\otimes \epsilon.$$
 Using this matrix we define a \emph{new} set of eight $\gamma$ matrices\footnote{The $\gamma$-matrices used below should be considered as new symbols, and should not be confused with the $\soa(4,4)$ $\gamma$-matrices in formulas defining $\tilde{\gamma}$-matrices at the beginning of this section. One should forget about the definition of $\tilde{\gamma}$s in the formula below.} by:
 $$\gamma_i\,=\,(i B+\mathrm{Id})\,\,\tilde{\gamma}_i\,\,(i B+\mathrm{Id})^{-1}, \quad\quad\forall i=1,\dots,8.$$
 One can check that these 8 matrices are \emph{all real} and that they satisfy the desired Dirac identity:
 $$\gamma_i\gamma_j+\gamma_j\gamma_i=2\delta_{ij} (I\otimes I\otimes I\otimes I), \quad i,j=1,\dots,8.$$
 Explicitly we have:
  $$\begin{aligned}
     &\gamma_1=\sigma_x\otimes\sigma_x\otimes\sigma_x\otimes\sigma_x\\
  &\gamma_2=-\epsilon\otimes\sigma_z\otimes\epsilon\otimes I\\
  &\gamma_3=\sigma_x\otimes\sigma_x\otimes\sigma_x\otimes\sigma_z\\
  &\gamma_4=\epsilon\otimes\sigma_z\otimes\sigma_x\otimes \epsilon\\
  &\gamma_5=\sigma_x\otimes\sigma_x\otimes\sigma_z\otimes I\\
  &\gamma_6=-\epsilon\otimes I\otimes \sigma_z\otimes \epsilon\\
    &\gamma_7=\sigma_x\otimes\sigma_z\otimes I\otimes I\\
     &\gamma_8=\sigma_x\otimes \epsilon\otimes \sigma_z\otimes \epsilon.
    \end{aligned}$$
The 28 generators of $\soa(8,0)$ in the Majorana-Dirac spinor representation $\rho$ in the space of Dirac spinors $S=\bbR^{16}$ are given by
$$\rho(A_{I(i,j)})=\tfrac12\gamma_i\gamma_j, \quad 1\leq i<j\leq 8,$$
where again we have used the function $I=I(i,j)$ defined in \eqref{Iij}. Note that since now $i<j$ can run from 1 to 8, the function has a range from 1 to 28.
We add to this the scaling generator, $\rho(A_{29})$, extending the Lie algebra $\soa(4,4)$ to $\mathfrak{coa}(4,4)$, given by
$$\rho(A_{29})=\tfrac12 I\otimes I\otimes I\otimes I.$$
In terms of the 2-dimensional Pauli matrices these generators look like:
\be\label{dir80}
\begin{array}{ll}
  \rho(A_1)= -\tfrac12 \sigma_z \otimes \epsilon\otimes \sigma_z\otimes \sigma_x,&\rho(A_{15})=-\tfrac12 \sigma_z \otimes \sigma_x\otimes I\otimes \epsilon,\\
  \rho(A_2)= \tfrac12  I \otimes I\otimes I\otimes \epsilon ,&\rho(A_{16})=\tfrac12 I \otimes \epsilon\otimes \sigma_x\otimes \sigma_x,\\
  \rho(A_3)= \tfrac12  \sigma_z \otimes \epsilon\otimes \sigma_z\otimes \sigma_z ,&\rho(A_{17})=\tfrac12  \sigma_z \otimes I  \otimes \epsilon\otimes I ,\\
  \rho(A_4)=\tfrac12   \sigma_z \otimes \epsilon\otimes I\otimes \sigma_z,&\rho(A_{18})=\tfrac12  I \otimes \epsilon\otimes \sigma_x\otimes \sigma_z ,\\
  \rho(A_5)= -\tfrac12  I \otimes I\otimes \sigma_z\otimes \epsilon,&\rho(A_{19})=-\tfrac12  \sigma_z \otimes I\otimes \sigma_x\otimes \epsilon ,\\
  \rho(A_6)= -\tfrac12  \sigma_z \otimes \epsilon\otimes I\otimes \sigma_x,&\rho(A_{20})=\tfrac12 I \otimes \epsilon\otimes \sigma_z\otimes I ,\\
  \rho(A_7)= \tfrac12  I \otimes I\otimes \epsilon\otimes \sigma_x,&\rho(A_{21})=\tfrac12 \sigma_z \otimes \sigma_z\otimes \sigma_z\otimes \epsilon,\\
  \rho(A_8)= -\tfrac12 \sigma_z \otimes \epsilon\otimes \sigma_x\otimes I,&\rho(A_{22})=\tfrac12 I \otimes \sigma_z\otimes \epsilon\otimes \sigma_z ,\\
  \rho(A_9)= \tfrac12 I \otimes I\otimes \epsilon\otimes \sigma_z,&\rho(A_{23})=-\tfrac12 \sigma_z \otimes \sigma_x\otimes \sigma_x\otimes \epsilon  ,\\
  \rho(A_{10})=\tfrac12  \sigma_z \otimes \epsilon\otimes \epsilon\otimes \epsilon,&\rho(A_{24})=-\tfrac12 I \otimes \sigma_z\otimes \epsilon\otimes \sigma_x ,\\
  \rho(A_{11})= -\tfrac12  \sigma_z \otimes \sigma_x\otimes \epsilon\otimes \sigma_z ,&\rho(A_{25})=-\tfrac12 \sigma_z \otimes \sigma_x\otimes \epsilon\otimes I  ,\\
  \rho(A_{12})= -\tfrac12 I \otimes \sigma_z\otimes \sigma_x\otimes \epsilon,&\rho(A_{26})=\tfrac12  I \otimes \sigma_z\otimes I\otimes \epsilon ,\\
  \rho(A_{13})=\tfrac12  \sigma_z \otimes \sigma_x\otimes \epsilon\otimes \sigma_x,&\rho(A_{27})=-\tfrac12  \sigma_z \otimes \epsilon\otimes I\otimes I,\\
   \rho(A_{14})= -\tfrac12  I \otimes \sigma_z\otimes \epsilon\otimes I,&\rho(A_{28})=-\tfrac12  I \otimes \sigma_x\otimes \sigma_z\otimes \epsilon.
\end{array}
\ee
Similarly as in the case of $\soa(4,4)$ this 16-dimensional representation of $\bbR\oplus\soa(4,4)$ is \emph{reducible}, again due to the appearance of $I$ and $\sigma_z$ only as the first factors in the above formulas. It \emph{splits} onto two real $8$-dimensional Weyl representations
$$\rho=\rho_+\oplus\rho_-\quad\mathrm{in}\quad S=S_+\oplus S_-, \quad \mathrm{dim}_\bbR S_\pm=8.$$

On generators of $\soa(8,0)$ these two 8-dimensional representations $\rho_\pm$, are given by:
\be
\begin{array}{ll}
  \rho_\pm(A_1)= \mp\tfrac12 \epsilon\otimes \sigma_z\otimes \sigma_x,&\rho_\pm(A_{15})=\mp\tfrac12  \sigma_x\otimes I\otimes \epsilon,\\
  \rho_\pm(A_2)= \tfrac12  I\otimes I\otimes \epsilon ,&\rho_\pm(A_{16})=\tfrac12 \epsilon\otimes \sigma_x\otimes \sigma_x,\\
  \rho_\pm(A_3)= \pm\tfrac12  \epsilon\otimes \sigma_z\otimes \sigma_z ,&\rho_\pm(A_{17})=\pm\tfrac12   I  \otimes \epsilon\otimes I ,\\
  \rho_\pm(A_4)=\pm\tfrac12   \epsilon\otimes I\otimes \sigma_z,&\rho_\pm(A_{18})=\tfrac12  \epsilon\otimes \sigma_x\otimes \sigma_z ,\\
  \rho_\pm(A_5)= -\tfrac12   I\otimes \sigma_z\otimes \epsilon,&\rho_\pm(A_{19})=\mp\tfrac12  I\otimes \sigma_x\otimes \epsilon ,\\
  \rho_\pm(A_6)= \mp\tfrac12  \epsilon\otimes I\otimes \sigma_x,&\rho_\pm(A_{20})=\tfrac12 \epsilon\otimes \sigma_z\otimes I ,\\
  \rho_\pm(A_7)= \tfrac12 I\otimes \epsilon\otimes \sigma_x,&\rho_\pm(A_{21})=\pm\tfrac12  \sigma_z\otimes \sigma_z\otimes \epsilon,\\
  \rho_\pm(A_8)= \mp\tfrac12 \epsilon\otimes \sigma_x\otimes I,&\rho_\pm(A_{22})=\tfrac12  \sigma_z\otimes \epsilon\otimes \sigma_z ,\\
  \rho_\pm(A_9)= \tfrac12  I\otimes \epsilon\otimes \sigma_z,&\rho_\pm(A_{23})=\mp\tfrac12  \sigma_x\otimes \sigma_x\otimes \epsilon  ,\\
  \rho_\pm(A_{10})=\pm\tfrac12  \epsilon\otimes \epsilon\otimes \epsilon,&\rho_\pm(A_{24})=-\tfrac12 \sigma_z\otimes \epsilon\otimes \sigma_x ,\\
  \rho_\pm(A_{11})= \mp\tfrac12  \sigma_x\otimes \epsilon\otimes \sigma_z ,&\rho_\pm(A_{25})=\mp\tfrac12 \sigma_x\otimes \epsilon\otimes I  ,\\
  \rho_\pm(A_{12})= -\tfrac12 \sigma_z\otimes \sigma_x\otimes \epsilon,&\rho_\pm(A_{26})=\tfrac12  \sigma_z\otimes I\otimes \epsilon ,\\
  \rho_\pm(A_{13})=\pm\tfrac12  \sigma_x\otimes \epsilon\otimes \sigma_x,&\rho_\pm(A_{27})=\mp\tfrac12   \epsilon\otimes I\otimes I,\\
   \rho_\pm(A_{14})= -\tfrac12  \sigma_z\otimes \epsilon\otimes I,&\rho_\pm(A_{28})=-\tfrac12  \sigma_x\otimes \sigma_z\otimes \epsilon.
\end{array}\label{weylso8}
\ee
We extend them em to $\bbR\oplus\soa(8,0)$ by adding
$$\rho_\pm(A_{29})=\tfrac12 I\otimes I\otimes I.$$
It follows that the Weyl representations $(\rho_\pm,S_\pm)$ of $\soa(4,4)$ are \emph{irreducible} and \emph{nonequivalent}.

We use them to find the defining representation $(\tau,R)$ of $\soa(8,0)$ in the vector space $R=\bbR^8$ of vectors. We again consider the tensor product representation $\rho_+\otimes\rho_-$. It decomposes as:
$$\rho_+\otimes\rho_-=\alpha\oplus\tau\quad\mathrm{in}\quad S_+\otimes S_-=T_{56}\oplus R,\quad\mathrm{with}\quad \dim_\bbR(R)=8,\,\,\dim_\bbR(T_{56})=56,$$
having irreducible components $(\alpha,T_{56})$ and $(\tau,R)$ of respective dimensions 56 and 8. Explicitly, on generators $A_I$ of $\bbR\oplus\soa(8,0)$, the 8-dimensional representation $\tau$ reads:
\be
\begin{array}{llll}
  \tau(A_1)=E_{38}-E_{83},&\quad  \tau(A_8)=E_{35}-E_{53},&\quad  \tau(A_{15})=E_{52}-E_{25},&\quad  \tau(A_{22})=E_{68}-E_{86},\\
  \tau(A_2)=E_{78}-E_{87}, &\quad\tau(A_9)=E_{75}-E_{57}, &\quad  \tau(A_{16})=E_{18}-E_{81},&\quad  \tau(A_{23})=E_{36}-E_{63}, \\
  \tau(A_3)=E_{37}-E_{73}, &\quad\tau(A_{10})=E_{54}-E_{45},&\quad  \tau(A_{17})=E_{31}-E_{13},&\quad  \tau(A_{24})=E_{76}-E_{67},\\
  \tau(A_4)=E_{84}-E_{48},&\quad  \tau(A_{11})=E_{28}-E_{82},&\quad  \tau(A_{18})=E_{71}-E_{17},&\quad  \tau(A_{25})=E_{64}-E_{46},\\
  \tau(A_5)=E_{43}-E_{34},&\quad  \tau(A_{12})=E_{32}-E_{23},&\quad \tau(A_{19})=E_{14}-E_{41},&\quad  \tau(A_{26})=E_{56}-E_{65},\\
 \tau(A_6)=E_{47}-E_{74},&\quad   \tau(A_{13})=E_{72}-E_{27},&\quad  \tau(A_{20})=E_{51}-E_{15},&\quad  \tau(A_{27})=E_{26}-E_{62},\\
 \tau(A_7)=E_{58}-E_{85}, &\quad \tau(A_{14})= E_{24}-E_{42},&\quad  
  \tau(A_{21})=E_{21}-E_{12},&\quad  \tau(A_{28})=E_{16}-E_{61},
\end{array}
\label{so8}\ee
where $E_{ij}$, $i,j=1,2,\dots,8$, denote $8\times 8$ matrices with zeroes everywhere except the value 1 in the entry $(i,j)$ seating at the crossing of the $i$th row and the $j$th column.

The three real, irreducible, pairwise nonequivalent representations $(\rho_+,\rho_-,\tau)$ of $\soa(8,0)$ given by the formulas \eqref{weylso8} and \eqref{so8} constitute the set of the \emph{triality representations} for $\soa(8,0)$.   

\section{Application: 2-step graded realizations of real forms of the exceptional Lie algebra $\mathfrak{e}_6$ }
The simple exceptional complex Lie algebra $\mathfrak{e}_6$ has the following \emph{noncompact} real forms
\begin{enumerate}
  \item $\mathfrak{e}_I$, with Satake diagram 
\tikzset{/Dynkin diagram/fold style/.style={stealth-stealth,thin,
    shorten <=1mm,shorten >=1mm}} \begin{dynkinDiagram}[edge length=.5cm]{E}{oooooo}
\end{dynkinDiagram},
\item $\mathfrak{e}_{II}$, with Satake diagram \tikzset{/Dynkin diagram/fold style/.style={stealth-stealth,thin,
    shorten <=1mm,shorten >=1mm}}
\begin{dynkinDiagram}[edge length=.5cm]{E}{oooooo}\invol{1}{6}\invol{3}{5}
\end{dynkinDiagram},
\item $\mathfrak{e}_{III}$, with Satake diagram 
\tikzset{/Dynkin diagram/fold style/.style={stealth-stealth,thin,
shorten <=1mm,shorten >=1mm}}\begin{dynkinDiagram}[edge length=.5cm]{E}{oo***o}\invol{1}{6}
\end{dynkinDiagram}, and 
 \item $\mathfrak{e}_{IV}$, with Satake diagram  \tikzset{/Dynkin diagram/fold style/.style={stealth-stealth,thin,
    shorten <=1mm,shorten >=1mm}}\begin{dynkinDiagram}[edge length=.5cm]{E}{o****o}
 \end{dynkinDiagram}.
\end{enumerate}
\`Elie Cartan in his theses \cite{CartanPhd, CartanPhdF} mentioned realization of the real form $\mathfrak{e}_I$ in $N=\bbR^{16}$. In the modern language, Cartan's realization is such that $\mathfrak{e}_I$ is the \emph{algebra of authomorphisms} of the flat model of a \emph{parabolic geometry} of type $(E_I,P)$, where the choice of parabolic subgroup in the real form $E_I$ of the exceptional Lie group ${\bf E}_6$ is indicated by the following decoration of the Satake diagram for $\mathfrak{e}_I$:  \tikzset{/Dynkin diagram/fold style/.style={stealth-stealth,thin,
    shorten <=1mm,shorten >=1mm}} \begin{dynkinDiagram}[edge length=.5cm]{E}{ooooot}
\end{dynkinDiagram}. The structure on the 16-dimensional manifold $N=E_I/P$, whose symmetry is $E_I$ is a Majorana-Weyl $\bbR{\bf Spin}(5,5)$ structure, i.e. the reduction of the structure group of the tangent bundle $\mathrm{T}N$ to the $\bbR{\bf Spin}(5,5)\subset\glg(16,\bbR)$ in the irreducible 16-dimensional representation of Majorana-Weyl spinors \cite{traut}. \emph{This geometry, as 1-step graded, is quite different from 2-step graded geometries considered in our paper}. We also mention, that if we wanted to have a realization of, say $\mathfrak{e}_{II}$ or $\mathfrak{e}_{III}$, in the spirit of Cartan's realization of $\mathfrak{e}_I$, i.e. if we crossed one lateral node in the Satake diagram of $\mathfrak{e}_{II}$ or $\mathfrak{e}_{III}$, we would be forced to cross the conjugated lateral root, resulting in the Satake diagrams  \tikzset{/Dynkin diagram/fold style/.style={stealth-stealth,thin,
    shorten <=1mm,shorten >=1mm}}
\begin{dynkinDiagram}[edge length=.5cm]{E}{toooot}\invol{1}{6}\invol{3}{5}
\end{dynkinDiagram} or \tikzset{/Dynkin diagram/fold style/.style={stealth-stealth,thin,
shorten <=1mm,shorten >=1mm}}\begin{dynkinDiagram}[edge length=.5cm]{E}{to***t}\invol{1}{6}
\end{dynkinDiagram},
which would give realizations of the respective $\mathfrak{e}_{II}$ and $\mathfrak{e}_{III}$ in dimension \emph{twenty four}. This we did in \cite{DJPZ} providing realizations of  $\mathfrak{e}_{II}$ and $\mathfrak{e}_{III}$ as Lie algebras of CR-authomorphisms of certain 24-dimensional CR manifolds of CR dimension 16, and CR (real) codimension 8. The important point of these realizations of these two real forms of $]\mathfrak{e}_e$ was that these geometries were 2-step graded, as in the case of Cartan's realization of $\mathfrak{f}_I$, and they could have been also thought as realizations in terms of the symmetry algebras of the structure $(M,{\mathcal D})$, where $M$ is a certain 24-dimensional real manifold, and $\mathcal D$ is a real rank 16-distribution on $M$ with $[{\mathcal D},{\mathcal D}]=\mathrm{T}M$. Thus these two geometries described by us in \cite {DJPZ} are 2-step graded geometries of distributions. Very much like Cartan's realization of $\mathfrak{f}_I$.

In this section we give the remaining similar realizations of the yet untreated cases of $\mathfrak{e}_{II}$ and $\mathfrak{e}_{III}$. 
\subsection{Realizations of $\mathfrak{e}_{I}$ and $\mathfrak{e}_{IV}$: generalities}
To get realizations  of $\mathfrak{e}_{I}$ and $\mathfrak{e}_{IV}$ in dimension 24, we decorate the Satake diagrams of these two Lie algebras as follows: \tikzset{/Dynkin diagram/fold style/.style={stealth-stealth,thin,
    shorten <=1mm,shorten >=1mm}} \begin{dynkinDiagram}[edge length=.5cm]{E}{toooot}
\end{dynkinDiagram} and  \tikzset{/Dynkin diagram/fold style/.style={stealth-stealth,thin,
    shorten <=1mm,shorten >=1mm}}\begin{dynkinDiagram}[edge length=.5cm]{E}{t****t}
\end{dynkinDiagram}.
These choices of a parabolic subalgebra in the respective $\mathfrak{e}_{I}$ and $\mathfrak{e}_{IV}$ produces the following gradation in these algebras:
$$\mathfrak{e}_A=\mathfrak{n}_{\minu2A}\oplus\mathfrak{n}_{\minu1A}\oplus\mathfrak{n}_{0A}\oplus\mathfrak{n}_{1A}\oplus\mathfrak{n}_{2A}\quad \mathrm{for}\quad A=I,IV,
$$
 with
 $$\mathfrak{n}_{\minu A}=\mathfrak{n}_{\minu2A}\oplus\mathfrak{n}_{\minu1A}\quad \mathrm{for}\quad A=I,IV,$$
 being 2-step nilpotent and having grading components $\mathfrak{n}_{-2A}$ and $\mathfrak{n}_{\minu1A}$ of respective dimension $r_A=8$ and $s_A=16$,
 $$r_A=\dim(\mathfrak{n}_{\minu2A})=8,\quad\quad s_A=\dim(\mathfrak{n}_{\minu1A})=16\quad \mathrm{for}\quad A=I,IV.$$
 The Lie algebra $\mathfrak{n}_{0A}$ in the Tanaka prolongation of $\mathfrak{n}_{\minu A}$ up to $0^{th}$ order is
 \begin{enumerate}
 \item $\mathfrak{n}_{0I}=2\bbR\oplus\soa(4,4)=\bbR\oplus\mathfrak{co}(4,4)$ in the case of $\mathfrak{e}_I$, and
   \item $\mathfrak{n}_{0IV}=2\bbR\oplus\soa(8,0)=\bbR\oplus\mathfrak{co}(8,0)$ in the case of $\mathfrak{e}_{IV}$.
 \end{enumerate}
 The last two statements, (1) and (2), get clear when one looks at the Satake diagrams we have just decorated. If we strip off the crossed nodes from these diagrams we get\begin{dynkinDiagram}[edge length=.4cm]{D}{oooo}\end{dynkinDiagram} and \begin{dynkinDiagram}[edge length=.4cm]{D}{****}
 \end{dynkinDiagram}, clearly the simple part of $\mathfrak{n}_{0A}$s above.

 Because of the grading property $[\mathfrak{n}_{iA},\mathfrak{n}_{jA}]\subset\mathfrak{n}_{(i+j)A}$ in the Lie algebras $\mathfrak{e}_A$, restricting to subalgebras $\mathfrak{n}_{\minu A}$ we see that we have representations $(\rho_A,\mathfrak{n}_{\minu1A})$ and $(\tau_A,\mathfrak{n}_{\minu 2A})$ given by the adjoint action of $\mathfrak{co}(4,4)$ or $\mathfrak{co}(8,0)$ which naturally seat in $\mathfrak{n}_{0A}$, respectively.

 There is no surprise that the representations $(\rho_A,\mathfrak{n}_{\minu1A})$ are the Dirac spinor representations \eqref{dir44} and \eqref{dir80} of the respective $\mathfrak{co}(4,4)$ and $\mathfrak{co}(8,0)$ parts of $\mathfrak{n}_{0A}$s in the 16-dimensional real vector spaces $\mathfrak{n}_{\minu1A}$. As such, these representations are \emph{reducible} and they split each $\mathfrak{n}_{0A}$, $A=I,IV$, onto two irreducible representations $(\rho_{A\pm},\mathfrak{n}_{\minu1A\pm})$ in real 8-dimensional spaces $\mathfrak{n}_{\minu1A\pm}$ of Weyl spinors. This shows that the  \emph{2-step nilpotent Lie algebra} $\mathfrak{n}_{\minu A}$ \emph{is}, for each $A=I,IV$, \emph{a natural representation space for the action of the three triality representations} $(\rho_+,\rho_-,\tau)$. We have,
 $$\begin{aligned}\mathfrak{n}_{\minu A}=&\mathfrak{n}_{\minu2A}\oplus\mathfrak{n}_{\minu1A}=\\
   &\mathfrak{n}_{\minu2A}\oplus\mathfrak{n}_{\minu1A+}\oplus\mathfrak{n}_{\minu1A-},\end{aligned}$$
 and the 8-dimensional real irreducible representations $(\tau_A,\rho_A+,\rho_A-)$ of $\mathfrak{co}(4,4)$ or $\mathfrak{co}(8,0)$ acting in the respective $\mathfrak{n}_{\minu2A}$, $\mathfrak{n}_{\minu1A+}$ and $\mathfrak{n}_{\minu1A-}$.

 We summarize the considerations from this section in the following theorem,
 \begin{theorem} (Natural realization of the triality representations)
   \begin{enumerate} \item The $\soa(4,4)$ triality:\\
     The real form $\mathfrak{e}_I$  of the simple exceptional Lie algebra $\mathfrak{e}_6$, when graded according to the following decoration of its Satake diagram \tikzset{/Dynkin diagram/fold style/.style={stealth-stealth,thin,
    shorten <=1mm,shorten >=1mm}} \begin{dynkinDiagram}[edge length=.5cm]{E}{toooot}
   \end{dynkinDiagram}, has the $\mathfrak{n}_{\minu}$ part as a real 24-dimensional vector space, naturally split onto the three real 8-dimensional components $\mathfrak{n}_{\minu2}$, $\mathfrak{n}_{\minu1+}$ and $\mathfrak{n}_{\minu1-}$,
   $$\mathfrak{n}_{\minu}=\mathfrak{n}_{\minu2}\oplus\mathfrak{n}_{\minu1+}\oplus\mathfrak{n}_{\minu1-}.$$
   This decomposition is $\soa(4,4)$ invariant and consists of components $\mathfrak{n}_{\minu2}$, $\mathfrak{n}_{\minu1+}$ and $\mathfrak{n}_{\minu1-}$, on which the triality representation 
   $$\tau\oplus\rho_+\oplus\rho_-$$
   of $\soa(4,4)$ acts irreducibly.
\item The $\soa(8,0)$ triality:\\
   Likewise, the real form $\mathfrak{e}_{IV}$  of the simple exceptional Lie algebra $\mathfrak{e}_6$, when graded according to the following decoration of its Satake diagram \tikzset{/Dynkin diagram/fold style/.style={stealth-stealth,thin,
    shorten <=1mm,shorten >=1mm}} \begin{dynkinDiagram}[edge length=.5cm]{E}{t****t}
   \end{dynkinDiagram}, has the $\mathfrak{n}_{\minu}$ part as a real 24-dimensional vector space, naturally split onto the three real 8-dimensional components $\mathfrak{n}_{\minu2}$, $\mathfrak{n}_{\minu1+}$ and $\mathfrak{n}_{\minu1-}$,
   $$\mathfrak{n}_{\minu}=\mathfrak{n}_{\minu2}\oplus\mathfrak{n}_{\minu1+}\oplus\mathfrak{n}_{\minu1-}.$$
   This decomposition is $\soa(8,0)$ invariant and consists of components $\mathfrak{n}_{\minu2}$, $\mathfrak{n}_{\minu1+}$ and $\mathfrak{n}_{\minu1-}$, on which the triality representation 
   $$\tau\oplus\rho_+\oplus\rho_-$$
   of $\soa(8,0)$ acts irreducibly.
   \end{enumerate}
   \end{theorem}
 \subsection{An explicit realization of $\mathfrak{e}_I$ in dimension 24}
 Taking as $(\rho,S)$ the Dirac spinors representation \eqref{dir44} of $\mathrm{co}(4,4)$ in dimension 16, and as $(\tau,R)$ the vectorial representation \eqref{tauweyl} of $\mathrm{co}(4,4)$ in dimension 8, we again are in the situation of a missing $\omega\in\mathrm{Hom}(\bigwedge^2S,R)$ from the triple $(\rho,\tau,\omega)$ described by the magical equation \eqref{maga}. Solving this equation for $\omega$ we obtain $\omega^i{}_{\mu\nu}$, $i=1,\dots,8$, $\mu,\nu=1,\dots,16$, which leads to the eight 2-forms $\omega^i=\tfrac12\omega^i{}_{\mu\nu}\der x^\mu\dz\der x^\nu$ on a 16-dimensional manifold $N=\bbR^{16}$, which read
 \be\begin{aligned}
   \omega^1=\,\,&-\der x^1\dz\der x^{10}+\der x^2\dz\der x^{9}+\der x^7\dz\der x^{16}-\der x^8\dz\der x^{15}\\
   \omega^2=\,\,&-\der x^2\dz\der x^{12}+\der x^4\dz\der x^{10}+\der x^6\dz\der x^{16}-\der x^8\dz\der x^{14}\\
   \omega^3=\,\,&-\der x^1\dz\der x^{12}+\der x^4\dz\der x^{9}+\der x^5\dz\der x^{16}-\der x^8\dz\der x^{13}\\
   \omega^4=\,\,&-\der x^5\dz\der x^{10}+\der x^6\dz\der x^{9}+\der x^7\dz\der x^{12}-\der x^8\dz\der x^{11}\\
   \omega^5=\,\,&-\der x^2\dz\der x^{11}+\der x^3\dz\der x^{10}+\der x^6\dz\der x^{15}-\der x^7\dz\der x^{14}\\
   \omega^6=\,\,&-\der x^1\dz\der x^{11}+\der x^3\dz\der x^{9}+\der x^5\dz\der x^{15}-\der x^7\dz\der x^{13}\\
   \omega^7=\,\,&-\der x^3\dz\der x^{12}+\der x^4\dz\der x^{11}+\der x^5\dz\der x^{14}-\der x^6\dz\der x^{13}\\
   \omega^8=\,\,&-\der x^1\dz\der x^{14}+\der x^2\dz\der x^{13}+\der x^3\dz\der x^{16}-\der x^4\dz\der x^{15}.
 \end{aligned}\label{2form3}\ee
 The manifold $N=\bbR^{16}$ with these 2-forms, after contactification, gives the following Theorem. 
\begin{theorem}\label{diste6}
   Let $M=\bbR^{24}$ with coordinates $(u^1,\dots,u^8,x^1,\dots ,x^{16})$, and consider eight 1-forms $\lambda^1,\dots,\lambda^8$ on $M$ given by
$$\begin{aligned}
   \lambda^1=\,\,&\der u^1-x^1\der x^{10}+ x^2\der x^{9}+ x^7\der x^{16}- x^8\der x^{15}\\
   \lambda^2=\,\,&\der u^2- x^2\der x^{12}+ x^4\der x^{10}+ x^6\der x^{16}- x^8\der x^{14}\\
   \lambda^3=\,\,&\der u^3- x^1\der x^{12}+ x^4\der x^{9}+x^5\der x^{16}- x^8\der x^{13}\\
   \lambda^4=\,\,&\der u^4- x^5\der x^{10}+ x^6\der x^{9}+ x^7\der x^{12}- x^8\der x^{11}\\
   \lambda^5=\,\,&\der u^5-x^2\der x^{11}+ x^3\der x^{10}+ x^6\der x^{15}- x^7\der x^{14}\\
   \lambda^6=\,\,&\der u^6- x^1\der x^{11}+ x^3\der x^{9}+ x^5\der x^{15}- x^7\der x^{13}\\
   \lambda^7=\,\,&\der u^7- x^3\der x^{12}+ x^4\der x^{11}+ x^5\der x^{14}- x^6\der x^{13}\\
   \lambda^8=\,\,&\der u^8- x^1\der x^{14}+ x^2\der x^{13}+ x^3\der x^{16}-x^4\der x^{15}
   \end{aligned}$$
      The rank 16 distribution ${\mathcal D}$ on $M$ defined as ${\mathcal D}=\{\mathrm{T}\bbR^{24}\ni X\,\,|\,\,X\hook\lambda^1=\dots=X\hook\lambda^8=0\}$ has its Lie algebra of infinitesimal authomorphisms $\mathfrak{aut}(\mathcal D)$ isomorphic to the Tanaka prolongation of $\mathfrak{n}_{\minu}=R\oplus S$, where $(\rho,S=\bbR^{16})$ is the Dirac spinors representation \eqref{dir44} of $\mathfrak{n}_{00}=\mathfrak{co}(4,4)$, and $(\tau,R=\bbR^8)$ is the vectorial representation \eqref{tauweyl} of $\mathfrak{n}_{00}$.

  The symmetry algebra $\mathfrak{aut}({\mathcal D})$ is isomorphic to the simple exceptional  Lie algebra $\mathfrak{e}_{I}$,
  $$\mathfrak{aut}({\mathcal D})=\mathfrak{e}_{I},$$
  having the following natural gradation 
  $$\mathfrak{aut}({\mathcal D})=\mathfrak{n}_{-2}\oplus\mathfrak{n}_{-1}\oplus\mathfrak{n}_0\oplus\mathfrak{n}_1\oplus\mathfrak{n}_2,$$
  with $\mathfrak{n}_{-2}=R$, $\mathfrak{n}_{-1}=S=S_+\oplus S_-$,
  $$
    \mathfrak{n}_0=\bbR\oplus\mathfrak{co}(4,4)\supset \mathfrak{n}_{00},$$
    $\mathfrak{n}_{1}=S^*$, $\mathfrak{n}_{2}=R^*$,
    and with the spaces $S_\pm$ being the carrier spaces for the Weyl spinors representations $\rho_\pm$ of $\mathfrak{co}(4,4)$.
 The gradation in $\mathfrak{e}_I$ is inherited from the distribution structure $(M,{\mathcal D})$. The duality signs $*$ at $R^*$ and $S^*$ above are with respect to the Killing form in $\mathfrak{e}_{I}$.

 The contactification $(M,{\mathcal D})$ is locally the flat model for the parabolic geometry of type $(E_{I},P_{I})$ related to the following \emph{crossed} Satake diagram  \tikzset{/Dynkin diagram/fold style/.style={stealth-stealth,thin,
    shorten <=1mm,shorten >=1mm}} \begin{dynkinDiagram}[edge length=.5cm]{E}{toooot}
\end{dynkinDiagram}.
  \end{theorem}
\begin{remark} 
Also the $\mathfrak{e}_I$ case, considered in this section, admits a description in terms of  an $\bbR\oplus\soa(4,4)$ invariant 4-form $\Phi$ in $S=\bbR^{16}$. 
  Now $\Phi$ may be represented by:
$$\Phi=h_{ij}\omega^i\dz\omega^j,$$
where $\omega^i$ are given by \eqref{2form3} and
$$\big(\,\,h_{ij}\,\,\big)\,\,=
\,\,\tfrac12\,\,\bma
0&0&0&0&0&0&1&0\\0&0&0&0&0&-1&0&0\\0&0&0&0&1&0&0&0\\0&0&0&0&0&0&0&1\\0&0&1&0&0&0&0&0\\0&-1&0&0&0&0&0&0\\1&0&0&0&0&0&0&0\\0&0&0&1&0&0&0&0 \ema.$$

Explicitly, the form $\Phi$ reads:
\be\begin{aligned}
  \Phi\,\,=\,\,&2\der x^1\dz\der x^2\dz\der x^{11}\dz\der x^{12}-2\der x^1\dz\der x^3\dz\der x^{10}\dz\der x^{12}+\\
  &2\der x^1\dz\der x^4\dz\der x^{10}\dz\der x^{11}+2\der x^1\dz\der x^5\dz\der x^{10}\dz\der x^{14}-\\&\der x^1\dz\der x^6\dz\der x^9\dz\der x^{14}-\der x^1\dz\der x^6\dz\der x^{10}\dz\der x^{13}-\\&\der x^1\dz\der x^6\dz\der x^{11}\dz\der x^{16}+\der x^1\dz\der x^6\dz\der x^{12}\dz\der x^{15}-\\
  &2\der x^1\dz\der x^7\dz\der x^{12}\dz\der x^{14}+2\der x^1\dz\der x^8\dz\der x^{11}\dz\der x^{14}+\\&2\der x^2\dz\der x^3\dz\der x^9\dz\der x^{12}-2\der x^2\dz\der x^4\dz\der x^9\dz\der x^{11}-\\&\der x^2\dz\der x^5\dz\der x^9\dz\der x^{14}-\der x^2\dz\der x^5\dz\der x^{10}\dz\der x^{13}+\\
  &\der x^2\dz\der x^5\dz\der x^{11}\dz\der x^{16}-\der x^2\dz\der x^5\dz\der x^{12}\dz\der x^{15}+\\
  &2\der x^2\dz\der x^6\dz\der x^{9}\dz\der x^{13}+2\der x^2\dz\der x^7\dz\der x^{12}\dz\der x^{13}-\\&2\der x^2\dz\der x^8\dz\der x^{11}\dz\der x^{13}+2\der x^3\dz\der x^4\dz\der x^{9}\dz\der x^{10}-\\&2\der x^3\dz\der x^5\dz\der x^{10}\dz\der x^{16}+2\der x^3\dz\der x^6\dz\der x^{9}\dz\der x^{16}+\\
  &2\der x^3\dz\der x^7\dz\der x^{12}\dz\der x^{16}-\der x^3\dz\der x^8\dz\der x^{9}\dz\der x^{14}+\\&\der x^3\dz\der x^8\dz\der x^{10}\dz\der x^{13}-\der x^3\dz\der x^8\dz\der x^{11}\dz\der x^{16}-\\&\der x^3\dz\der x^8\dz\der x^{12}\dz\der x^{15}+2\der x^4\dz\der x^5\dz\der x^{10}\dz\der x^{15}-\\
  &2\der x^4\dz\der x^6\dz\der x^{9}\dz\der x^{15}+\der x^4\dz\der x^7\dz\der x^{9}\dz\der x^{14}-\\
  &\der x^4\dz\der x^7\dz\der x^{10}\dz\der x^{13}-\der x^4\dz\der x^7\dz\der x^{11}\dz\der x^{16}-\\&\der x^4\dz\der x^7\dz\der x^{12}\dz\der x^{15}+2\der x^4\dz\der x^8\dz\der x^{11}\dz\der x^{15}+\\&2\der x^5\dz\der x^6\dz\der x^{15}\dz\der x^{16}-2\der x^5\dz\der x^7\dz\der x^{14}\dz\der x^{16}+\\
  &2\der x^5\dz\der x^8\dz\der x^{14}\dz\der x^{15}+2\der x^6\dz\der x^7\dz\der x^{13}\dz\der x^{16}-\\&2\der x^6\dz\der x^8\dz\der x^{13}\dz\der x^{15}+2\der x^7\dz\der x^8\dz\der x^{13}\dz\der x^{14}.
 \end{aligned}
\label{4formfiii}\ee
This 4-form is such that its stabilizer in $\gla(16,\bbR)$ is $\mathfrak{n}_0=\bbR\oplus\mathfrak{co}(4,4)$. When restricted to $\mathfrak{n}_{00}=\mathfrak{co}(4,4)$ this stabilizer is given precisely in the Mayorana Dirac spinor representation
$$\rho=\rho_+\oplus\rho_-$$
as in \eqref{dir44}-\eqref{rhopm}. 
\end{remark}

\subsection{An explicit realization of $\mathfrak{e}_{IV}$ in dimension 24}
Similarly as in the previous section we take as $(\rho,S)$ the Dirac spinors representation \eqref{dir80} of $\mathrm{co}(8.0)$ in dimension 16, and as $(\tau,R)$ the vectorial representation \eqref{so8} of $\mathrm{co}(8,0)$ in dimension 8and we search for  $\omega\in\mathrm{Hom}(\bigwedge^2S,R)$ solving the magical equation \eqref{maga}. We obtain $\omega^i{}_{\mu\nu}$, $i=1,\dots,8$, $\mu,\nu=1,\dots,16$, which provides us with the eight 2-forms $\omega^i=\tfrac12\omega^i{}_{\mu\nu}\der x^\mu\dz\der x^\nu$ on a 16-dimensional manifold $N=\bbR^{16}$, which read
 \be\begin{aligned}
   \omega^1=\,\,&\der x^1\dz\der x^{9}+\der x^2\dz\der x^{10}+\der x^3\dz\der x^{11}+\der x^4\dz\der x^{12}-\\&\der x^5\dz\der x^{13}-\der x^6\dz\der x^{14}-\der x^7\dz\der x^{15}-\der x^8\dz\der x^{16}\\
   \omega^2=\,\,&-\der x^1\dz\der x^{10}+\der x^2\dz\der x^{9}+\der x^3\dz\der x^{12}-\der x^4\dz\der x^{11}-\\&\der x^5\dz\der x^{14}+\der x^6\dz\der x^{13}+\der x^7\dz\der x^{16}-\der x^8\dz\der x^{15}\\
   \omega^3=\,\,&-\der x^1\dz\der x^{11}-\der x^2\dz\der x^{12}+\der x^3\dz\der x^{9}+\der x^4\dz\der x^{10}+\\&\der x^5\dz\der x^{15}+\der x^6\dz\der x^{16}-\der x^7\dz\der x^{13}-\der x^8\dz\der x^{14}\\
   \omega^4=\,\,&-\der x^1\dz\der x^{12}+\der x^2\dz\der x^{11}-\der x^3\dz\der x^{10}+\der x^4\dz\der x^{9}+\\&\der x^5\dz\der x^{16}-\der x^6\dz\der x^{15}+\der x^7\dz\der x^{14}-\der x^8\dz\der x^{13}\\
   \omega^5=\,\,&\der x^1\dz\der x^{13}+\der x^2\dz\der x^{14}-\der x^3\dz\der x^{15}-\der x^4\dz\der x^{16}+\\&\der x^5\dz\der x^{9}+\der x^6\dz\der x^{10}-\der x^7\dz\der x^{11}-\der x^8\dz\der x^{12}\\
   \omega^6=\,\,&\der x^1\dz\der x^{14}-\der x^2\dz\der x^{13}-\der x^3\dz\der x^{16}+\der x^4\dz\der x^{15}-\\&\der x^5\dz\der x^{10}+\der x^6\dz\der x^{9}+\der x^7\dz\der x^{12}-\der x^8\dz\der x^{11}\\
   \omega^7=\,\,&\der x^1\dz\der x^{15}-\der x^2\dz\der x^{16}+\der x^3\dz\der x^{13}-\der x^4\dz\der x^{14}+\\&\der x^5\dz\der x^{11}-\der x^6\dz\der x^{12}+\der x^7\dz\der x^{9}-\der x^8\dz\der x^{10}\\
   \omega^8=\,\,&-\der x^1\dz\der x^{16}-\der x^2\dz\der x^{15}-\der x^3\dz\der x^{14}-\der x^4\dz\der x^{13}-\\&\der x^5\dz\der x^{12}-\der x^6\dz\der x^{11}-\der x^7\dz\der x^{10}-\der x^8\dz\der x^{9}.
 \end{aligned}\label{2form4}\ee
Contactifying, we have the following theorem

\begin{theorem}
   Let $M=\bbR^{24}$ with coordinates $(u^1,\dots,u^8,x^1,\dots ,x^{16})$, and consider eight 1-forms $\lambda^1,\dots,\lambda^8$ on $M$ given by

 $$\begin{aligned}
   \lambda^1=\,\,&\der u^1+ x^1\der x^{9}+ x^2\der x^{10}+ x^3 \der x^{11}+ x^4 \der x^{12}- x^5 \der x^{13}- x^6 \der x^{14}- x^7 \der x^{15}- x^8 \der x^{16}\\
   \lambda^2=\,\,&\der u^2- x^1 \der x^{10}+ x^2 \der x^{9}+ x^3 \der x^{12}- x^4 \der x^{11}-  x^5 \der x^{14}+ x^6 \der x^{13}+ x^7 \der x^{16}- x^8 \der x^{15}\\
   \lambda^3=\,\,&\der u^3- x^1 \der x^{11}- x^2 \der x^{12}+ x^3 \der x^{9}+ x^4 \der x^{10}+  x^5 \der x^{15}+ x^6 \der x^{16}- x^7 \der x^{13}- x^8 \der x^{14}\\
   \lambda^4=\,\,&\der u^4- x^1 \der x^{12}+ x^2 \der x^{11}- x^3 \der x^{10}+ x^4 \der x^{9}+  x^5 \der x^{16}- x^6 \der x^{15}+ x^7 \der x^{14}- x^8 \der x^{13}\\
   \lambda^5=\,\,&\der u^5+ x^1 \der x^{13}+ x^2 \der x^{14}- x^3 \der x^{15}- x^4 \der x^{16}+  x^5 \der x^{9}+ x^6 \der x^{10}- x^7 \der x^{11}- x^8 \der x^{12}\\
   \lambda^6=\,\,&\der u^6+ x^1 \der x^{14}- x^2 \der x^{13}- x^3 \der x^{16}+ x^4 \der x^{15}-  x^5 \der x^{10}+ x^6 \der x^{9}+ x^7 \der x^{12}- x^8 \der x^{11}\\
   \lambda^7=\,\,&\der u^7+ x^1 \der x^{15}- x^2 \der x^{16}+ x^3 \der x^{13}- x^4 \der x^{14}+  x^5 \der x^{11}- x^6 \der x^{12}+ x^7 \der x^{9}- x^8 \der x^{10}\\
   \lambda^8=\,\,&\der u^8 - x^1 \der x^{16}- x^2 \der x^{15}- x^3 \der x^{14}- x^4 \der x^{13}-  x^5 \der x^{12}- x^6 \der x^{11}- x^7 \der x^{10}- x^8 \der x^{9}.
   \end{aligned}$$
   
      The rank 16 distribution ${\mathcal D}$ on $M$ defined as ${\mathcal D}=\{\mathrm{T}\bbR^{24}\ni X\,\,|\,\,X\hook\lambda^1=\dots=X\hook\lambda^8=0\}$ has its Lie algebra of infinitesimal authomorphisms $\mathfrak{aut}(\mathcal D)$ isomorphic to the Tanaka prolongation of $\mathfrak{n}_{\minu}=R\oplus S$, where $(\rho,S=\bbR^{16})$ is the Dirac spinors representation \eqref{dir44} of $\mathfrak{n}_{00}=\mathfrak{co}(8,0)$, and $(\tau,R=\bbR^8)$ is the vectorial representation \eqref{tauweyl} of $\mathfrak{n}_{00}$.

  The symmetry algebra $\mathfrak{aut}({\mathcal D})$ is isomorphic to the simple exceptional  Lie algebra $\mathfrak{e}_{IV}$,
  $$\mathfrak{aut}({\mathcal D})=\mathfrak{e}_{IV},$$
  having the following natural gradation 
  $$\mathfrak{aut}({\mathcal D})=\mathfrak{n}_{-2}\oplus\mathfrak{n}_{-1}\oplus\mathfrak{n}_0\oplus\mathfrak{n}_1\oplus\mathfrak{n}_2,$$
  with $\mathfrak{n}_{-2}=R$, $\mathfrak{n}_{-1}=S=S_+\oplus S_-$,
  $$
    \mathfrak{n}_0=\bbR\oplus\mathfrak{co}(8,0)\supset \mathfrak{n}_{00},$$
    $\mathfrak{n}_{1}=S^*$, $\mathfrak{n}_{2}=R^*$,
    and with the spaces $S_\pm$ being the Carrier spaces for the Weyl spinors representations $\rho_\pm$ of $\mathfrak{co}(8,0)$.
 The gradation in $e_{iV}$ is inherited from the distribution structure $(M,{\mathcal D})$. The duality signs $*$ at $R^*$ and $S^*$ above are with respect to the Killing form in $\mathfrak{e}_{IV}$.

 The contactification $(M,{\mathcal D})$ is locally the flat model for the parabolic geometry of type $(E_{IV},P_{IV})$ related to the following \emph{crossed} Satake diagram   \tikzset{/Dynkin diagram/fold style/.style={stealth-stealth,thin,
    shorten <=1mm,shorten >=1mm}}\begin{dynkinDiagram}[edge length=.5cm]{E}{t****t}
\end{dynkinDiagram}.
\end{theorem}
\begin{remark} 
Again we have a description of the relevant representations in terms of  an $\bbR\oplus\mathfrak{co}(8,0)$ invariant 4-form $\Phi$ in $S=\bbR^{16}$. 
  Now $\Phi$ may be represented by:
$$\Phi=h_{ij}\omega^i\dz\omega^j,$$
where $\omega^i$ are given by \eqref{2form4} and
$$\big(\,\,h_{ij}\,\,\big)\,\,=\,\,
\bma
1&0&0&0&0&0&0&0\\0&1&0&0&0&0&0&0\\0&0&1&0&0&0&0&0\\0&0&0&1&0&0&0&0\\0&0&0&0&1&0&0&0\\0&0&0&0&0&1&0&0\\0&0&0&0&0&0&1&0\\0&0&0&0&0&0&0&1 \ema.
$$
This 4-form is such that its stabilizer in $\gla(16,\bbR)$ is $\mathfrak{n}_0=\bbR\oplus\mathfrak{co}(8,0)$. When restricted to $\mathfrak{n}_{00}=\mathfrak{co}(8,0)$ this stabilizer is given precisely in the Mayorana Dirac spinor representation
$$\rho=\rho_+\oplus\rho_-$$
as in \eqref{dir80}-\eqref{weylso8}. 
\end{remark}
\section{Application: one more realization of $\mathfrak{e}_6$ and a realization of $\mathfrak{b}_6$}
Between the 24-dimensional realizations of $\mathfrak{e}_6$ mentioned in this paper, and  Cartan's 16-dimensional realization of $\mathfrak{e}_6$ associated with the grading\tikzset{/Dynkin diagram/fold style/.style={stealth-stealth,thin,
    shorten <=1mm,shorten >=1mm}}\begin{dynkinDiagram}[edge length=.5cm]{E}{ooooot}\end{dynkinDiagram}, there are 21-dimensional realizations of this algebra $\mathfrak{e}_6$ associated with the following Dynkin diagram crossing
\tikzset{/Dynkin diagram/fold style/.style={stealth-stealth,thin,
    shorten <=1mm,shorten >=1mm}}\begin{dynkinDiagram}[edge length=.5cm]{E}{otoooo}\end{dynkinDiagram}. These define contact $\mathfrak{e}_6$ geometries and are described in \cite{CS} p. 425-426.

\subsection{Realization of $\mathfrak{e}_I$ in dimension 25} 
Here we will briefly discuss yet another realization, now in dimension 25, corresponding to the following Dynkin diagram crossing: \tikzset{/Dynkin diagram/fold style/.style={stealth-stealth,thin,
    shorten <=1mm,shorten >=1mm}}\begin{dynkinDiagram}[edge length=.5cm]{E}{ooooto}\end{dynkinDiagram} of $\mathfrak{e}_6$. This is for example mentioned in \cite{weyman}. Looking at the Satake diagrams of real forms of $\mathfrak{e}_6$, we see that this realization is only possible for the real form $\mathfrak{e}_I$.

So we again use our Corollary \ref{cruco} with now $\mathfrak{n}_{00}=\sla(2,\bbR)\oplus\sla(5,\bbR)$ and with representations $(\rho,S)$ and $(\tau,R)$, as indicated in \cite{weyman} Section 5.3, 
$S=\bbR^2\otimes\bigwedge^2\bbR^5,\quad R=\bigwedge^2\bbR^2\otimes\bigwedge^4\bbR^5.$

To be more explicit we obtain this representations as follows:
\begin{itemize}
\item We start with the defining representations $\tau_2$ of $\sla(2,\bbR)$ in $\bbR^2$ and $\tau_5$ of $\sla(5,\bbR)$ in $\bbR^5$, and define the representation $$\rho=\tau_2\otimes\big(\tau_5\dz\tau_5\big)\quad\mathrm{of}\quad \sla(2,\bbR)\oplus\sla(5,\bbR)\quad\mathrm{ in}\quad \textstyle S=\bbR^2\otimes\bigwedge^2\bbR^5=\bbR^{20}.$$
  The representation $(\rho,S)$ is an irreducible real 20-dimensional representation  of
  $$\mathfrak{n}_{00}=\sla(2,\bbR)\oplus\sla(5,\bbR).$$
\item Then we decompose the $190$-dimensional representation $\rho\dz\rho$ onto the irreducibles:
  $$\textstyle \rho\dz\rho=\alpha\oplus\tau\oplus\beta\quad\mathrm{in}\quad\bigwedge_{50}\oplus R\oplus\bigwedge_{135},$$
  with $(\alpha,\bigwedge_{50})$ being 50-dimensional, $(\tau,R)$ being 5-dimensional, and $(\beta,\bigwedge_{135})$ being $135$-dimensional.
  \item We take the 20-dimensional representation $(\rho,S)$ and the $5$-dimensional representation $(\tau,R)$ of $\mathfrak{n}_{00}=\sla(2,\bbR)\oplus\sla(5,\bbR)$ as above, and apply our Corollary \ref{cruco}.
    \end{itemize}

We obtain the following theorem.
\begin{theorem}
   Let $M=\bbR^{25}$ with coordinates $(u^1,\dots,u^5,x^1,\dots ,x^{20})$, and consider five 1-forms $\lambda^1,\dots,\lambda^5$ on $M$ given by
 $$\begin{aligned}
   \lambda^1=\,\,&\der u^1- x^3\der x^{20}+ x^5\der x^{19}- x^6 \der x^{18}- x^8 \der x^{16}+ x^9 \der x^{15}- x^{10} \der x^{13}\\
   \lambda^2=\,\,&\der u^2- x^2 \der x^{20}+ x^4 \der x^{19}- x^6 \der x^{17}- x^7 \der x^{16}+  x^9 \der x^{14}- x^{10} \der x^{12}\\
   \lambda^3=\,\,&\der u^3- x^1 \der x^{20}+ x^4 \der x^{18}-x^5 \der x^{17}- x^7 \der x^{15}+  x^8 \der x^{14}- x^{10} \der x^{11}\\
   \lambda^4=\,\,&\der u^4- x^1 \der x^{19}+ x^2 \der x^{18}- x^3 \der x^{17}- x^7 \der x^{13}+  x^8 \der x^{12}- x^9 \der x^{11}\\
   \lambda^5=\,\,&\der u^5- x^1 \der x^{16}+ x^2 \der x^{15}- x^3 \der x^{14}- x^4 \der x^{13}+  x^5 \der x^{12}- x^6 \der x^{11}.
   \end{aligned}$$
         The rank 20 distribution ${\mathcal D}$ on $M$ defined as ${\mathcal D}=\{\mathrm{T}\bbR^{25}\ni X\,\,|\,\,X\hook\lambda^1=\dots=X\hook\lambda^5=0\}$ has its Lie algebra of infinitesimal authomorphisms $\mathfrak{aut}(\mathcal D)$ isomorphic to the Tanaka prolongation of $\mathfrak{n}_{\minu}=R\oplus S$, where $(\rho,S=\bbR^{20})$ is the 20-dimensional irreducible representation of $\mathfrak{n}_{00}=\sla(2,\bbR)\oplus\sla(5,\bbR)$, and $(\tau,R=\bbR^5)$ is the 5-dimensional irreducible subrepresentation $\tau\in(\rho\dz\rho)$ of $\mathfrak{n}_{00}$.

  The symmetry algebra $\mathfrak{aut}({\mathcal D})$ is isomorphic to the simple exceptional  Lie algebra $\mathfrak{e}_{I}$,
  $$\mathfrak{aut}({\mathcal D})=\mathfrak{e}_{I},$$
  having the following natural gradation 
  $$\mathfrak{aut}({\mathcal D})=\mathfrak{n}_{-2}\oplus\mathfrak{n}_{-1}\oplus\mathfrak{n}_0\oplus\mathfrak{n}_1\oplus\mathfrak{n}_2,$$
  with $\mathfrak{n}_{-2}=R$, $\mathfrak{n}_{-1}=S$,
  $$
    \mathfrak{n}_0=\bbR\oplus\sla(2,\bbR)\oplus\sla(5,\bbR)\supset \mathfrak{n}_{00},$$
    $\mathfrak{n}_{1}=S^*$, $\mathfrak{n}_{2}=R^*$.
 The gradation in $e_I$ is inherited from the distribution structure $(M,{\mathcal D})$. The duality signs $*$ at $R^*$ and $S^*$ above are with respect to the Killing form in $\mathfrak{e}_{I}$.

 The contactification $(M,{\mathcal D})$ is locally the flat model for the parabolic geometry of type $(E_{I},P_{I*})$ related to the following \emph{crossed} Satake diagram   \tikzset{/Dynkin diagram/fold style/.style={stealth-stealth,thin,
    shorten <=1mm,shorten >=1mm}}\begin{dynkinDiagram}[edge length=.5cm]{E}{ooooto}
\end{dynkinDiagram}.
\end{theorem}
\subsection{A realization of $\soa(7,6)$ in dimension 21} We know from\cite{CS} that the crossed Satake diagram \tikzset{/Dynkin diagram/fold style/.style={stealth-stealth,thin,
    shorten <=1mm,shorten >=1mm}}\begin{dynkinDiagram}[edge length=.5cm]{E}{otoooo}\end{dynkinDiagram} corresponds to the $\mathfrak{e}_I$-symmetric contact geometry in dimension 21. It corresponds to the grading
$$\mathfrak{e}_I=\mathfrak{n}_{-2}\oplus\mathfrak{n}_{-1}\oplus\mathfrak{n}_0\oplus\mathfrak{n}_1\oplus\mathfrak{n}_2,$$
with $\dim(\mathfrak{n}_{\pm 1})=20$, $\dim(\mathfrak{n}_{\pm2})=1$ and $\mathfrak{n}_0=\gla(6,\bbR)$.

Interestingly dimension $n=78$ is the dimension not only of the exceptional simple Lie algebra $\mathfrak{e}_6$, but also for the  \emph{simple} Lie algebras $\mathfrak{b}_6$ and $\mathfrak{c}_6$. For example, if we take the crossed Satake diagram \tikzset{/Dynkin diagram/fold style/.style={stealth-stealth,thin,
    shorten <=1mm,shorten >=1mm}}\begin{dynkinDiagram}[edge length=.5cm]{B}{ooooot}\end{dynkinDiagram} we describe the following gradation
$$\soa(7,6)=\mathfrak{n}_{-2}\oplus\mathfrak{n}_{-1}\oplus\mathfrak{n}_0\oplus\mathfrak{n}_1\oplus\mathfrak{n}_2,$$
with $\dim(\mathfrak{n}_{\pm 1})=6$, $\dim(\mathfrak{n}_{\pm2})=15$ and $\mathfrak{n}_0=\gla(6,\bbR)$, in the simple Lie algebra $\soa(7,6)$. Here, taking  $(\rho,S)$ as the defining representation $\rho(A)=A$ of $\glg(6,\bbR)$ in $S=\bbR^6$, and taking the representation $(\tau,R)$ to be  $\tau=\rho\dz\rho$ in $R=\bigwedge^2\bbR^6=\bbR^{15}$, and applying our Corollary \ref{cruco} we get the following theorem\footnote{We invoke it, just to show that we do not only use spin representations in this paper.}.
\begin{theorem}
   Let $M=\bbR^{21}$ with coordinates $(u^1,\dots,u^{15},x^1,\dots ,x^{6})$, and consider fifteen 1-forms $\lambda^1,\dots,\lambda^5$ on $M$ given by
 $$
   \lambda^{I(i,j)}=\,\,\der u^{I(i,j)}- x^i\der x^{j},$$
   with $$I(i,j)=1+i+\tfrac12(j-3)j,\quad 1\leq i<j\leq 6.$$
         The rank 6 distribution ${\mathcal D}$ on $M$ defined as ${\mathcal D}=\{\mathrm{T}\bbR^{21}\ni X\,\,|\,\,X\hook\lambda^1=\dots=X\hook\lambda^{15}=0\}$ has its Lie algebra of infinitesimal authomorphisms $\mathfrak{aut}(\mathcal D)$ isomorphic to the Tanaka prolongation of $\mathfrak{n}_{\minu}=R\oplus S$, where $(\rho,S=\bbR^{6})$ is the 6-dimensional defining representation of $\mathfrak{n}_{00}=\gla(6,\bbR)$, and $(\tau,R=\bigwedge^2\bbR^6)$ is the $15$-dimensional irreducible subrepresentation $\tau=\rho\dz\rho$ of $\mathfrak{n}_{00}$.

  The symmetry algebra $\mathfrak{aut}({\mathcal D})$ is isomorphic to the simple exceptional  Lie algebra $\soa(7,6)$,
  $$\mathfrak{aut}({\mathcal D})=\soa(7,6),$$
  having the following natural gradation 
  $$\mathfrak{aut}({\mathcal D})=\mathfrak{n}_{-2}\oplus\mathfrak{n}_{-1}\oplus\mathfrak{n}_0\oplus\mathfrak{n}_1\oplus\mathfrak{n}_2,$$
  with $\mathfrak{n}_{-2}=R$, $\mathfrak{n}_{-1}=S$,
  $$
    \mathfrak{n}_0=\gla(6,\bbR)= \mathfrak{n}_{00},$$
    $\mathfrak{n}_{1}=S^*$, $\mathfrak{n}_{2}=R^*$.
 The gradation in $\soa(7,6)$ is inherited from the distribution structure $(M,{\mathcal D})$. The duality signs $*$ at $R^*$ and $S^*$ above are with respect to the Killing form in $\soa(7,6)$.

 The contactification $(M,{\mathcal D})$ is locally the flat model for the parabolic geometry of type $(\soa(7,6),P)$ related to the following \emph{crossed} Satake diagram   \tikzset{/Dynkin diagram/fold style/.style={stealth-stealth,thin,
    shorten <=1mm,shorten >=1mm}}\begin{dynkinDiagram}[edge length=.5cm]{B}{ooooot}
\end{dynkinDiagram}.
\end{theorem}

\end{document}